\documentclass[12pt]{amsart}
\usepackage[a4paper, top=30mm, bottom=40mm, inner=24mm, outer=24mm]{geometry}
\usepackage{comment}
\excludecomment{confidential}
\usepackage{tikz-cd}
\usepackage{url}
\usepackage{amsmath}
\usepackage{amssymb}
\usepackage{amsthm}
\usepackage{amsthm}
\usepackage{letltxmacro}
\usepackage{changepage}
\usepackage{latexsym}
\usepackage{graphicx}
\usepackage{hyperref}
\usepackage{enumerate}
\usepackage{amsfonts}
\usepackage{amstext}
\usepackage{mathrsfs}
\usepackage{float}
\usepackage[colorinlistoftodos, textwidth=20mm]{todonotes}
\newtheorem{thm}{Theorem}
\usepackage{amsthm}
\makeatletter
\def\th@plain{%
  \thm@notefont{}
  \itshape 
}
\def\th@definition{%
  \thm@notefont{}
  \normalfont 
}
\makeatother
\LetLtxMacro\oldproof\proof
\let\endoldproof\endproof
\renewenvironment{proof}[1][\proofname]
  {\begin{adjustwidth}{2em}{2em}
   \oldproof[#1]}
  {\endoldproof
   \end{adjustwidth}\vspace{3mm}}
\newtheorem{cor}[thm]{Corollary}
\newtheorem{lem}[thm]{Lemma}
\newtheorem{prop}[thm]{Proposition}
\theoremstyle{definition}
\newtheorem{defn}[thm]{Definition}
\newtheorem{rem}[thm]{Remark}

\newcommand{\lpi}{\langle}
\newcommand{\rpi}{\rangle}

\newcommand{\fa}{\mathfrak{a}}

\newcommand{\fe}{\mathfrak{e}}

\newcommand{\bba}{\mathbb{A}}

\newcommand{\bbc}{\mathbb{C}}

\newcommand{\bbg}{\mathbb{G}}

\newcommand{\bbn}{\mathbb{N}}

\newcommand{\bbr}{\mathbb{R}}

\newcommand{\bbt}{\mathbb{T}}

\newcommand{\bbz}{\mathbb{Z}}

\newcommand{\Cg}{\mathcal{G}}

\newcommand{\uB}{\underline{B}}
\newcommand{\uC}{\underline{C}}

\newcommand{\uH}{\underline{H}}

\newcommand{\uZ}{\underline{Z}}

\newcommand{\cpv}{\begin{center}\begin{minipage}[t]{14cm}\small{\it Proof.} }
\newcommand{\fpv}{\fim\end{minipage}\end{center}}
\newcommand{\fim}{\hfill $\Box$}

\newcommand{\ga}{{\Gamma_{E/F}}}
\newcommand{\we}{{W_{E/F}}}
\newcommand{\ef}{{E/F}}

\newcommand{\LT}{{}^{\mathsf{L}}T}

\newcommand{\dT}{\widehat{T}}
\newcommand{\dX}{\widehat{X}}

\newcommand{\logtc}{\widetilde{\log}_T}

\DeclareMathOperator{\Hom}{Hom}

\DeclareMathOperator{\Aut}{Aut}

\DeclareMathOperator{\id}{id}
\DeclareMathOperator{\Fr}{Fr}

\DeclareMathOperator{\Pl}{Pl}
\DeclareMathOperator{\Lie}{Lie}
\DeclareMathOperator{\Res}{Res}
\DeclareMathOperator{\Inf}{Inf}

\begin{document}
\title{A remark on the Langlands correspondence for tori}
\author{Marcelo De Martino}
    \address{Forward College, Rua das Flores, 71, 1200-194 Lisboa, Portugal}
    \email{marcelo.demartino@forward-college.eu}
\author{Eric Opdam}
    \address{Korteweg-de Vries Institute for Mathematics, Universiteit van Amsterdam, Science Park 107, 1090 GE Amsterdam, the Netherlands}
   \email{e.m.opdam@uva.nl}
 \thanks{MDM was partly supported by the special research fund (BOF) from Ghent University [BOF20/PDO/058].}
\subjclass[2020]{11R34, 20J06, 22E55}

\begin{abstract}
For an algebraic torus defined over a local (or global) field $F$, a celebrated result of R.P. Langlands establishes a natural homomorphism from the group of continuous cohomology classes of the Weil group, valued in the dual torus, onto the space of complex characters of the rational points of the torus (or automorphic characters in the global case).
We expand on this result by detailing its topological 
aspects. We show that if we topologize the relevant spaces of continuous homomorphisms and continuous cochains using the compact-open topology, Langlands's map becomes a (surjective, 
finite-to-one) homomorphism of abelian  complex  Lie groups.
Moreover, we demonstrate that, in both the local and global settings, the subset of unramified characters is the identity component of the relevant space of characters.
Finally, we compare the group of unramified characters with the Galois
(co)invariants of the dual torus.
\end{abstract}

\maketitle

\tableofcontents

\section{Introduction}
Let  $T$ be an algebraic torus defined over a field $F$ (all fields in this article will either be local or global) and splitting over a finite Galois extension $E/F$, with Galois group  $\ga$.
Denote by $X=\Hom(T_E,\bbg_m)$ 
the character lattice, by $\dX = \Hom(X,\bbz)$ the cocharacter lattice and by $\dT = \Hom(\dX,\bbc^*)$ the complex dual torus of $T$. Whenever $F$ is global, we let $\bba_F$ denote its ring of ad\`eles. 

In \cite{La}, Langlands described the space of continuous characters 
$T(F)\to\bbc^*$ (local case) and of 
automorphic characters $T(\bba_F)/T(F)\to\bbc^*$ (global case). Let
\begin{equation}\label{eq:C}
C_E = \left\{
\begin{array}{ll}
E^* & \textup{if }E\textup{ local},\\
\bba^*_E/E^* & \textup{if }E\textup{ global},
\end{array}
\right.
\end{equation}
and denote by $\we$ the relative Weil group, defined via the extension
\begin{equation}\label{eq:WeilGroup}
  0 \longrightarrow C_E\longrightarrow \we \stackrel{p}{\longrightarrow} \ga \longrightarrow 0  
\end{equation}
 of $\ga$ by $C_E$, defined by the fundamental class in $H^2(\Gamma_{E/F},C_E)$ (see \cite[(1.2)]{Ta2}). In what follows, we shall denote the image of the map  $\we \stackrel{p}{\longrightarrow} \ga$ by $p(\omega)=\overline{\omega}$. Langlands showed that there is a finite-to-one map from the group $H^1_c(\we,\dT)$, of continuous first-cohomology classes of $\we$ with values in $\dT$, onto the relevant space of continuous characters. When $F$ is local, the correspondence is one-to-one; when $F$ is global, it is one-to-one modulo locally trivial equivalence $\sim_{l.e.}$ (see Definition \ref{def:LocTrivClasses}, below).
 
The aim of this article is to show that, when suitably topologizing the cohomology groups and the spaces of characters, Langlands's canonical map  is continuous and open, and that the space of unramified characters, $X_{T}$ (see Definitions \ref{def:LocUnrChar} and \ref{def:GlobUnrChar}), is the identity component of the relevant space of characters. Our first main result is the following. Given our torus $T$ defined over $F$, let
\begin{equation}\label{eq:X}
\mathsf{X}_\bbc(T,F) = \left\{
\begin{array}{ll}
\Hom_c(T(F),\bbc^*) & \textup{if }F\textup{ local},\\
\Hom_c(T(\bba_F)/T(F),\bbc^*) & \textup{if }F\textup{ global}.
\end{array}
\right.
\end{equation}
In both cases, $\mathsf{X}_\bbc(T,F)$ is topologized with the compact-open topology, in which case it becomes a metrizable space \cite[Section 8]{Ar}. 
We note that $X_T\subset \mathsf{X}_\bbc(T,F)$ is 
a closed subset. We equip $X_T$ with the subspace 
topology. 
\begin{thm}\label{thm:mainThm0}
Let $F$ be a local or global field. Then $\mathsf{X}_\bbc(T,F)$ is an abelian complex Lie group 
with identity component $X_T$ and Lie algebra 
$\Lie(\dT)^{\Gamma_\ef}$. 
If $V_F\cong \bbr$ then $X_T$ is a complex vector space, and if $V_F\cong\bbz$ then $X_T$ is a complex algebraic torus. 
\end{thm}

The proof of this theorem is given in Corollary \ref{cor:main0}, below. The main point is to show that $X_T\subset \mathsf{X}_\bbc(T,F)$ is open. 
The basic structure of $X_T$ as a complex vector space or complex algebraic torus is well known, and is recalled in more detail in Corollary \ref{cor:GlobXT}. 

We now turn our attention to the cohomological side of Langlands's correspondence. Using a formalism of C. Moore \cite{Mo1,Mo2,Mo3,Mo4} we consider the theory of measurable cohomology groups $\uH^*(W_\ef, \dT)$ (see Section \ref{s:CohoSide}), which provides fruitful interaction between topological and homological algebra features. We then show that $H^1_c(\we,\dT)$ naturally has
the structure of a complete metric space (see Theorem \ref{thm:Hausdorff}) and prove the following result.
\begin{thm}\label{thm:mainThm}
For $F$ a local or global field, $H^1_c(\we,\dT)$ is an 
abelian complex Lie group
with Lie algebra $H^1_c(\we,\Lie(\dT))\cong\Lie(\dT)^{\Gamma_\ef}$. If $V_F\cong \bbr$ then 
the identity component of $H^1_c(\we,\dT)$ is a complex vector space, 
and if $V_F\cong \bbz$ then the identity component of $H^1_c(\we,\dT)$
is a complex algebraic torus. The canonical surjective homomorphism
\[
H^1_c(\we,\dT) \to \mathsf{X}_\bbc(T,F)
\]
established by Langlands, is a homomorphism of complex Lie groups with finite kernel, and is an isomorphism if $F$ is local.
\end{thm}

We shall discuss Langlands's map in the next section (see Theorem \ref{thm:LangTori}) and provide the arguments proving Theorem \ref{thm:mainThm} in Propositions \ref{prop:OpClosed}, \ref{prop:H1middle}, \ref{prop:ConCompOp}, Theorem \ref{thm:Rcase}, Corollary \ref{cor:Amax} 
and Theorem \ref{thm:ApexThm}, below. 

We emphasize that Langlands’s canonical homomorphism arises through a sequence of identifications involving class field theory and homological algebra, and its topological properties are not manifest from the construction itself. With the Lie group structures at hand,  
the proof that Langlands’s map is a Lie group homomorphism reduces to 
the verification of linearity of the associated map at the level of 
Lie algebras, as we will see in Theorem \ref{thm:ApexThm}.

We now turn to our final results, which provide structural information on $\mathsf{X}_\bbc(T,F)$ and on 
$H^1_c(W_{\ef},\dT)$ as applications of Theorem \ref{thm:mainThm0} and Theorem \ref{thm:mainThm}.
The fundamental exponential sequence 
$0\to \bbz\to \bbc\to \bbc^* \to 0$ gives rise 
to an exact sequence (via tensoring with the lattice $X$)
\begin{equation}
0\to X\to \Lie(\dT) \to \dT \to 0, 
\end{equation}
from which we derive (see \eqref{eq:LES}, where the next term $H^1(\Gamma_{\ef},X)$ is torsion)
\begin{equation}
0\to X^{\Gamma_\ef}\to \Lie(\dT)^{\Gamma_{E/F}}\to (\dT^{\Gamma_{E/F}})^\circ \to 0,
\end{equation}
where $(\dT^{\Gamma_\ef})^\circ$ is the identity component of the $\ga$-invariants of $\dT$.
The vector space $\Lie(\dT)$ admits a canonical decomposition $\Lie(\dT)=\Lie(\dT)^{\Gamma_{E/F}}\oplus N$ where $N$ denotes the direct sum of the nontrivial 
irreducible $\Gamma_{E/F}$-subrepresentations. Let 
$\textup{Pr}_{\Gamma_\ef}$ denote the projection onto $\Lie(\dT)^{\Gamma_{E/F}}$. Then
\begin{equation}
0\to \textup{Pr}_{\Gamma_\ef}(X)\to \Lie(\dT)^{\Gamma_{E/F}}\to \dT_{\Gamma_{E/F}} \to 0, 
\end{equation}
and we have the following results:
\begin{thm}\label{thm:mainThm2}
Let $F$ be a local or global field, and let 
$T$ be an algebraic torus over $F$. Let $E/F$ be a Galois extension such that $T$ splits over $E$.
\begin{enumerate}
\item[(a)] The identity component $H_c^1(W_\ef,\dT)^\circ$ of $H_c^1(W_\ef,\dT)$  satisfies:
\begin{enumerate}
\item[(i)]
If $V_F\cong \bbr$, then 
$H_c^1(W_\ef,\dT)^\circ\approx X_T\approx \Lie(\dT)^{\Gamma_\ef}$ as $\bbc$-vector spaces.
\item[(ii)]
If $F$ is local non-Archimedean, we have, as complex algebraic tori,  
\begin{equation*}
H_c^1(W_\ef,\dT)^\circ
\approx X_T\approx ((\dT^{N_F})_\textup{Fr})^\circ,
\end{equation*} where $N_F$ is the image of the inertia subgroup in $W_{\ef}$ and $\Fr$ is a Frobenius automorphism. 
\item[(iii)]
If $F$ is global and $V_F\cong \bbz$, then 
$H_c^1(W_\ef,\dT)^\circ$ is a complex algebraic torus with  cocharacter lattice $A\subset \bbr\otimes X^{\Gamma_{E/F}}\subset \Lie(\dT)^\ga$,  which satisfies:
\begin{equation*}\label{eq:sandwich}
X^{\Gamma_{\ef}} \subset A \subset X_*(X_T) \subset
\textup{Pr}_{\Gamma_\ef}(X).
\end{equation*}
\end{enumerate}
\item[(b)] The component group of $H_c^1(W_\ef,\dT)$ satisfies, as topological groups, 
\begin{equation*}
    H_c^1(W_\ef,\dT)/H_c^1(W_\ef,\dT)^\circ\approx \uH^2(W_\ef,X). 
\end{equation*}
\item[(c)] Langlands's isomorphism 
\begin{equation*}
\Lambda: H_c^1(W_\ef,\dT)\to  \Hom_c(\Hom_{\Gamma_\ef}(X,C_E),\bbc^*)
\end{equation*}
is an isomorphism of complex Lie groups  
and induces an isomorphism of discrete groups: 
\begin{equation*}
\uH^2(W_\ef,X)\approx \Hom_c(\Hom_{\Gamma_\ef}(X,C^1_E),\bbc^*).
\end{equation*}
\end{enumerate}
\end{thm}

\begin{rem}\label{rem:Schwein0}
    When $F$ is local non-Archimedean, Theorem \ref{thm:mainThm2}(c) yields  $\uH^2(W_\ef,X)\approx\Hom_c(T(F)^1,\bbc^*)$, recovering an observation made by Schwein in \cite[Section 5.2]{Sc} (see also Remark \ref{rem:Schwein}, below). 
\end{rem}

To prove Theorem \ref{thm:mainThm2}(a) we need Theorem \ref{thm:mainThm} together with key results we prove later on. In the case when $V_F\cong \bbr$, it follows from Theorem \ref{thm:Rcase} and the fact that a vector space admits no non-trivial finite subgroups. In the case when $V_F\cong \bbz$, we need Theorem \ref{thm:Zcase} and Propositions \ref{prop:LnAHaines} and \ref{prop:ExplicitCocycle}. The specific description for $F$ local non-Archimedean is dealt with in Proposition \ref{prop:LnAHaines} (see also \cite[Section 3.3.1]{Ha}). 
Theorem \ref{thm:mainThm2}(b),(c) are proven in (the proof of) Proposition \ref{prop:H2X}. 

\subsection{Notations} Given two topological groups $G,H$ we shall write $G\cong H$ to mean that $G$ is isomorphic to $H$ as groups and $G\approx H$ to mean that $G$ is isomorphic to $H$ as topological groups. Also, we shall denote by $G^\circ$ the identity component of the topological group $G$.

 \subsection{Acknowledgements} We thank A. Kret for insightful discussions on this topic and M. Solleveld for suggesting the reference \cite{Ha} and for pointing out an inaccuracy in an earlier version. We are grateful to the anonymous referee for a careful reading of the manuscript and for their valuable suggestions, which have greatly improved the clarity and quality of the final version.

\section{Langlands's parametrization}
We start by recalling Langlands's parametrization  (see also the works of Birkbeck \cite{Bi}, Labesse \cite{Lab} and Yu \cite{Yu}, this last one in the local case). We retain the notation regarding the extension $E/F$ as above. There is a well-known equivalence of categories \cite[8.12]{Bo}
\[
\left\{
\begin{array}{c}
\textup{Algebraic tori } T/F\\
\textup{that split over } E
\end{array}
\right\}\longleftrightarrow
\left\{
\begin{array}{c}
\textup{Lattices } X \textup{ with} \\
\Gamma_{E/F}\textup{-action}\\
\end{array}
\right\}
\]
given by sending $T\mapsto X = \Hom(T_E,\bbg_m)$. 
Under this correspondence, we have $T(E) = \Hom(X,E^*)$ 
and $T(F)=\Hom_\ga(X,E^*)$. Denote by $\dX := \Hom(X,\bbz)$ and $\dT := \Hom(\dX,\bbc^*)$. 
We define an action of the Weil group $\we$ on $\dT$  via the projection $p:\we\to\ga$; in particular, $C_E$ acts trivially.

We shall write $\LT := \dT\rtimes \ga$ and we say that a homomorphism $\varphi:\we \to \LT$ is an \emph{$L$-homomorphism} if  by composing with the projection to the second factor we get $\textup{pr}_2\varphi = p$. We recall \cite[p. 232]{La} that the space of Langlands parameters of $T$, denoted $\Phi(T)$, is defined as the group of all continuous $L$-homomorphisms $\varphi:\we\to\LT$, modulo $\dT$ conjugation. 

\begin{defn}
For a locally compact group $\Cg$ endowed with a continuous action on $\dT$, let $Z_c^1(\Cg,\dT)$ denote the (abelian) group of continuous $1$-cocycles of $\Cg$ with values in $\dT$ and let $B_c^1(\Cg,\dT)$ denote the subgroup of continuous $1$-coboundaries. We topologize these spaces with the compact-open topology, and we let the quotient $H^1_c(\Cg,\dT)$ denote the group of continuous $1$-cohomology classes, endowed with the quotient topology.
\end{defn}

As is well-known \cite[p. 232]{La} we have the following result.

\begin{prop}[\cite{La}]\label{prop:ParCoh}
The space $\Phi(T)$ is isomorphic to $H^1_c(\we,\dT)$ as abelian groups.
\end{prop}

\begin{thm}[\cite{La}] \label{thm:LangTori}
Let $F$ be local or global. There are canonical homomorphisms 
\[
H^1_c(\we,\dT) \stackrel{\Lambda}{\to}\Hom_c(\Hom_{\ga}(X,C_E),\bbc^*)\stackrel{\rho}{\to} \mathsf{X}_\bbc(T,F),
\]
where $\Lambda$ is an isomorphism and $\rho$ surjective. 
If $F$ is local, then $\rho$ is also an isomorphism. If $F$ is global, then the kernel of $\rho$ is finite and 
consists of the cohomology classes which are locally trivial. 
\end{thm}

We now describe the locally trivial classes. For a local or global field $F$, let $(W_F,\varphi,\{r_E\})$ be the absolute Weil group of $F$ (see \cite{Ta2}). Here, the index $E$ varies over all finite extensions of $F$ inside a fixed separable closure $\bar F/F$, $\varphi$ is a continuous map with dense image $W_F\to\Gamma_F$ (where $\Gamma_F$ is the absolute Galois group), and $r_E:C_E\to W_E^{\textup{ab}}$ is the reciprocity isomorphism, where for each finite extension $\ef$ we have $W_E =  \varphi^{-1}(\Gamma_E)$. 

When $F$ is global, let $\Pl(F)$ denote the set of places of $F$. From the local-global relationship \cite[(1.6.1)]{Ta2}, for each $v\in \Pl(F)$, the inclusion $i_v:\bar F \to \bar F_v$
induces a unique (up to inner automorphisms of $W_F$) continuous map $\theta_v:W_{F_v}\to W_F$ such that $\varphi\theta_v=i_v\varphi_v$,
where we also denote by $i_v:\Gamma_{F_v}\to \Gamma_F$ the induced map between the absolute Galois groups of $F_v$ and $F$. We have the following (topological version of a) well-known result:

\begin{prop}\label{prop:LocTriv}
For $F$ local or global, pulling-back via the quotient map $W_{F}\to W_{E/F}$ yields an isomorphism of topological groups
$\Inf:H_c^1(\we, A) \stackrel{\approx}{\to} H_c^1(W_F,A)$ for any Hausdorff $W_F$-module $A$ in which $W_E$ acts trivially.
\end{prop}
\begin{proof} The argument in \cite[Proposition 3.1]{Bi} shows that 
the inflation map is a bijective homomorphism.  Recall that  $W_{\ef} \cong W_F/W_E^c$, where $W_E^c$ denotes the closure of the subgroup generated by all
commutators in $W_E$.
Since $W_E$ is isomorphic to the (semi-)direct product of its normal subgroup $W_E^1$ and 
$V_E\subset \mathbb{R}_+$, 
where $W_E^1$ denotes the kernel of the absolute value homomorphism $W_E \to \mathbb{R}_+$ (see \cite[(1.4.6)]{Ta2} and also Section~\ref{sec:absolutevals}, below), 
it follows that that $W_E^c\subset W_E^1$. 
In particular, $W_E^c$ is compact.
Hence, the compact-open topologies defined on $H_c^1(\we,A)$ and $H_c^1(W_F,A)$
coincide under this bijection, turning it into a  homeomorphism.
\end{proof}

Pulling-back via the continuous maps $\theta_v:W_{F_v}\to W_F$, for each $v$, yields the map 
\[
\prod_{v\in \Pl(F)}\Res_v:H_c^1(W_F,\dT) \to \prod_{v\in\Pl(F)} H_c^1(W_{F_v},\dT).
\]

\begin{defn}\label{def:LocTrivClasses}
We say that a Langlands parameter $\varphi \in \Phi(T)\cong H^1_c(\we,\dT)$ is \emph{locally trivial} if its image in $\prod_{v\in\Pl(F)} H_c^1(W_{F_v},\dT)$ is trivial.
\end{defn}

\section{Remarks on the compact-open topology}

In this paper we will deal a lot with the compact-open topology on the space of continuous functions between two topological spaces, so we gather in this section some relevant information about this topology, for convenience of the reader.

Let $A, B$ be topological spaces. Given any $K\subset A$ compact and $W\subset B$ open, we let $V(K,W) = \{f:A\to B \textup{ continuous }\mid f(K)\subset W\}$. The compact-open topology on the space $\mathscr{C}(A,B)$ of continuous functions between $A$ and $B$ is the topology generated by the sub-basis $\{V(K,W)\mid K \subset A\textup{ compact and }W\subset B\textup{ open}\}$. 

Furthermore, recall that a topological space $A$ is called hemicompact if there is a sequence $K_1, K_2,\ldots$ of compact subsets of $A$ such that $A = \cup_{n=1}^\infty K_n$ and any compact subset $K$ of $A$ is contained in a finite union $K_{n_1}\cup\cdots\cup K_{n_p}$ of  compact sets in that sequence. The multiplicative groups $C_F$ (with $F$ local or global) and finite extensions of them are examples of hemicompact topological spaces.

\begin{prop}\label{p:CompOpen}
The compact-open topology on $\mathscr{C}(A,B)$ satisfies the following properties.
\begin{itemize}
    \item[(a)] If $B$ is a metric space, then a sequence $(f_n)_{n\in \bbz_{\geq 0}}$ in $\mathscr{C}(A,B)$ converges in the compact-open topology to a continuous function $f$ if, and only if, $f_n$ converges uniformly to $f\in \mathscr{C}(A,B)$ on any compact subset $K$ of $A$.
    \item[(b)] If $B$ is a metric space, then $\mathscr{C}(A,B)$ is metrizable if and only if $A$ is hemicompact.
    \item[(c)] If $B$ is a metric space and $A$ is hemicompact, then the compact-open topology on $\mathscr{C}(A,B)$ is entirely
determined by sequential uniform convergence on compact sets.
    \item[(d)] If $A, B$ are topological groups with $A$  locally compact Hausdorff and $B$ Hausdorff, then the space of continuous homomorphisms $\Hom_c(A,B)$ is a closed subspace of $\mathscr{C}(A,B)$.
    \item[(e)] If $A, B_1, B_2$ are topological groups with $A$ locally compact Hausdorff, and $B_1, B_2$ metric spaces then
    \[
    \Hom_c(A,B_1\times B_2) \approx \Hom_c(A,B_1) \times \Hom_c(A,B_2).
    \]
    \item[(f)] If $A_1,A_2,B$ are topological groups with $A_1,A_2$ locally compact Hausdorff and $B$ an abelian group
with a metric topology, then
    \[
    \Hom_c(A_1\times A_2,B) \approx \Hom_c(A_1,B) \times \Hom_c(A_2,B).
    \]
\end{itemize}
\end{prop}

\begin{proof}
    Item (a) is \cite[Theorem 6]{Ar} and (b) \cite[Theorems 7 and 8]{Ar}. Item (c) is a consequence of (a) and (b). For item (d), note that $\Hom_c(A,B) = F^{-1}(\{c_e\})$ where $c_e$ is the constant homomorphism $c_e:A\times A\to B$ that maps $(a_1,a_2)\mapsto  e_B$  for all $(a_1,a_2)\in A\times A$ and where
 $F:\mathscr{C}(A,B) \to \mathscr{C}(A\times A,B)$ is given by
    \[
    F(f)(a_1,a_2) = f(a_1)f(a_2)f(a_1a_2)^{-1}.
    \] 
    The assumption on $B$ implies that $\{c_e\}$ is closed (see \cite[Theorem 1]{Ar}), while the assumption on $A$ implies that evaluation map $\mathscr{C}(A,B)\times A \to B,(f,a)\mapsto f(a)$ is continuous  (see \cite[Theorem 2]{Ar}), from which $F$
    is continuous, since $F = \mu_*\circ\epsilon$ with $\epsilon:\mathscr{C}(A,B) \to \mathscr{C}(A\times A,B\times B\times B)$ continuous and given by
    \[
    \epsilon(f)(a_1,a_2) = (f(a_1),f(a_2),f(a_1a_2)),
    \]   
    and $\mu_*:\mathscr{C}(A\times A,B\times B\times B) \to \mathscr{C}(A\times A,B)$ is the continuous map 
    \[g\mapsto \mu_*(g)(a_1,a_2) = \mu(g(a_1,a_2)),\] 
    obtained by composing with  $\mu:B\times B \times B \to B,(b_1,b_2,b_3)\mapsto b_1b_2b_3^{-1}$.
    
    For (e), it is shown in \cite[Chapitre X, \textsection1.4, Corollaire 1]{Bou} that $\mathscr{C}(A,B_1\times B_2)$ is homeomorphic to $\mathscr{C}(A,B_1) \times \mathscr{C}(A,B_2)$ under the milder condition that $B_1,B_2$ are uniform spaces and $A$ is any topological space, so the claim in this item follows from (d). For (f) we adapt the argument of \cite[Theorem (23.18)]{HR} in the case when $B$ is the unit circle. Let $\theta:\Hom_c(A_1,B) \times \Hom_c(A_2,B) \to \Hom_c(A_1\times A_2,B)$ be defined by
    \[
    \theta(f_1,f_2)(a_1,a_2) = f_1(a_1)f_2(a_2).
    \]
    This map is an isomorphism of  groups (see the first paragraph of the proof of \cite[Theorem (23.18)]{HR}). Now, given $K\subset A_1\times A_2$ compact and $\varepsilon>0$, let $V(K,\varepsilon)$ denote a neighborhood of the trivial homomorphism in $\Hom_c(A_1\times A_2,B)$ whose elements send $K$ into $B(e_B,\varepsilon)$, the open ball of radius $\varepsilon$ around the identity element $e_B\in B$. Similarly, by using the canonical projections, let $K_j = \textup{pr}_j(K) \subset A_j$ for $j = 1,2$ and let $V(K_j,\varepsilon)$ be the open subset of $\Hom_c(A_j,B)$ of all homomorphisms sending $K_j$ into $B(e_B,\varepsilon)$. Then, as the metric in $B$ can be assumed to be left-invariant (see \cite[Theorem (8.3)]{HR}), we obtain that $\theta(V(K_1,\varepsilon/2)\times V(K_2,\varepsilon/2))\subset V(K,\varepsilon)$. Conversely, if $V(K_j,\varepsilon_j)$ is an arbitrary neighborhood of the trivial homomorphism in $\Hom_c(A_j,B)$, with $K_j$ compact, for $j=1,2$, by setting $K = K_1\cup\{e_{A_1}\}\times K_2\cup\{e_{A_2}\}$, then
    \[
    V(K,\min(\varepsilon_1,\varepsilon_2)) \subset \theta(V(K_1,\varepsilon_1)\times V(K_2,\varepsilon_2))
    \]
    which shows that $\theta$ is a homeomorphism.
    \end{proof}   

\section{Unramified characters}\label{s:unrchar}

In this section, we will recall the definitions of unramified characters of a torus and establish some facts about them. We start by recalling the relevant absolute values in the groups we are interested in.

\subsection{Absolute values}\label{sec:absolutevals}
Suppose first that $F$ is a local field endowed with a normalized absolute value $|\cdot|_F:F\to \bbr_{\geq 0}$ with respect to which $F$ is a complete topological field (and locally compact). When $F$ is non-Archimedean, we let $\kappa_F$ denote its residue field, $q_F = \#\kappa_F$ and let $v_F:F\to \bbz$ denote its discrete valuation. If $F$ is Archimedean (i.e., $F$ is either $\bbr$ or $\bbc$) let $|\alpha|$ denote the usual absolute value in these fields. In each case, the normalized absolute value satisfies
\begin{equation}\label{eq:NormAbsVal}
|\alpha|_F = \left\{
\begin{array}{rl}
|\alpha|, &\textup{if }F=\bbr,\\
|\alpha|^2, &\textup{if }F=\bbc,\\
q_F^{-v_F(\alpha)}, &\textup{if }F\textup{ is non-Archimedean}.
\end{array}
\right.
\end{equation}

In the case when $F$ is a global field with ring of ad\`eles $\bba_F$, for each $v\in\Pl(F)$ we let $|\cdot|_v$ be the normalized absolute value of the local completion $F_v$ of $F$, as in (\ref{eq:NormAbsVal}). Here, any subscript referring to a local field $F_v$ will be only denoted by $v$. If $\alpha = (a_v)_v \in \bba_F$, let  $|\alpha|_F = \prod_v|a_v|_v$ denote the adelic absolute value, which restricts to a homomorphism $|\cdot|_F:\bba^*_F\to \bbr_+$ and induces a homomorphism $C_F\to\bbr_+$, by Artin's product formula. 

\begin{defn}
    For $F$ either local or global, the absolute value homomorphism\footnote{Although the notation \( |\cdot|_{\mathbb{A}_F} \) is customary for the adelic absolute value in the global case, we use \( |\cdot|_F \) to align with the unified notation for the groups \( C_F \).} is denoted by $|\cdot|_F:C_F\to \bbr_+$. We define $C^1_F$ and $V_F$ to be the kernel and the image of $|\cdot|_F$.
\end{defn}

We now summarize the properties of these topological groups.

\begin{prop}\label{p:CFsplitting}
The following assertions hold true.
\begin{itemize}
    \item[(a)] $C_F$ is isomorphic to $V_F\times C_F^1$ both algebraically and topologically. 
    \item[(b)] $C_F^1$ is a compact subgroup of $C_F$.
    \item[(c)] When $F$ is a number field or a local Archimedean field, we have $V_F=\bbr_+$.
    \item[(d)] When $F$ is a global function field or a local non-Archimedean field, we have  $V_F = q_F^{\bbz}$, where $q_F$ is the cardinality of the residue field in the local case and $q_F$ is the cardinality of the field of constants in the global case. 
\end{itemize}
\end{prop}

\begin{proof}
    This is rather well known. We refer to \cite[Theorem 6, Chapter IV]{We} for the algebraic part of (a). In the global case, the continuous section splitting the sequence topologically is canonically defined in the number field case \cite[Corollary 2, p. 75]{We} and not canonical in the function field case \cite[Corollary 1, p. 75]{We}. We refer to \cite[Theorem 5-15]{RV} for (b) and \cite[Theorem 5-14]{RV} for (c) and (d) (see also \cite[Chapter VII, \textsection 2]{NSW}).
\end{proof}

\begin{rem}\label{rem:GalSplitting}
    We note that under the finite Galois extension $E/F$, the norm $|\cdot|_E$ is $\Gamma_\ef$-invariant. The (algebraic and topological) splitting $C_E = C_E^1\times V_E$ of Proposition \ref{p:CFsplitting}(a) is not, in general, a splitting as $\Gamma_\ef$-modules. For example, in the local non-Archimedean case, a $\Gamma_\ef$-splitting exists if and only if $E/F$ is an unramified
    extension. On the other hand, when $V_F\cong\bbr$, we always have a $\Gamma_\ef$-splitting. In the global number field, the splitting described in \cite[Chapter IV \textsection 4]{We}, Corollary 2 to Theorem 5 and Theorem 6, is $\ga$-invariant since the action of the Galois group permutes the Archimedean places.
\end{rem}

Finally, if $F$ is a local or global field, recall (see \cite{Ta2}) that the absolute Weil group $W_F$ comes equipped with an absolute value extending the normalized absolute value $|\cdot|_F$ on $C_F \cong W_F^{\textup{ab}}=W_F/W_F^{\textup{c}}$ (here, $W_F^{\textup{c}}$ means the closure of the subgroup generated by all commutators in $W_F$). Since for any extension $E/F$ we have $W_E^{\textup{c}}\subseteq W_F^{\textup{c}}$, the absolute value on $W_F$ descends
to an absolute value on $\we \cong W_F/W_E^{\textup{c}}$, denoted $|\cdot|_{\ef}$. We let $W^1_\ef$ denote the kernel. Clearly, the image of $|\cdot|_\ef$ is $V_F$, the same as the image of $|\cdot|_F$.

\subsection{Unramified characters of multiplicative groups}
Given a field $F$, local or global, we say $F$ is of $\bbz$- or $\bbr$-type if, respectively, we have $V_F\cong\bbz$ or $V_F\cong \bbr$. The following definition and proposition appear in \cite[Sections 2.3 and 4.4]{Ta1}.

\begin{defn}
An \emph{unramified character} of  $C_F$ is a continuous homomorphism $C_F\to \bbc^*$ which is trivial on $C^1_F$.
\end{defn}

\begin{prop}[\cite{Ta1}]
Any unramified character of $C_F$ is of the type $\chi:\alpha \mapsto |\alpha|_F^s$ for some $s\in \bbc$. If $F$ is of $\bbr$-type, then $s$ is uniquely defined from $\chi$. If $F$ is of $\bbz$-type, $s$ is uniquely defined modulo $2\pi i /\log(q_F)$.
\end{prop}
We now recall how to extend this notion of unramified characters of the multiplicative groups $C_F$ to any torus.

\subsection{Unramified characters: local setting}
In this section, $F$ is a local field and $T$ is a torus defined over  $F$. If $X = \Hom(T_E,\bbg_m)$ is the character lattice of $T$
and $\ef$ is a Galois
extension splitting $T$, recall (e.g., from \cite[Section 8.11]{Bo}) that $X^{\Gamma_{\ef}}$ is the lattice of rational characters of $T$.
Note that each $\chi \in X^\ga$ defines a continuous homomorphism 
$|\chi|_F:T(F)\to V_F\subseteq \bbr_+$ via $|\chi|_F:t \mapsto |\chi(t)|_F$.
Let $T(F)^1 = \cap_\chi\ker(|\chi|_F)$, with $\chi \in X^\ga$.

\begin{defn}\label{def:LocUnrChar}
A continuous homomorphism $T(F) \to \bbc^*$ is called an \emph{unramified character} if it is trivial on $T(F)^1$.
Denote by $X_{T}$ the group of all unramified characters of $T$. 
\end{defn}

The group $T(F)^1$ plays a crucial role in this theory, and it is important to further characterize it. As we will see below, this group is the maximal compact subgroup of $T(F)$. However, we will delay its study until after we describe the space of complex characters of the group $\Hom_{\Gamma_\ef}(X,C_E)$, as this latter group can be treated in a more unified manner for both local and global fields. 

\subsection{Unramified characters: global setting}\label{s:GlobUnramif}

In the global case, given a torus $T$ defined over a global field $F$ and splitting over a finite Galois extension $E/F$, 
any element $\chi \in X^\ga$ defines, for each $v\in \Pl(F)$, an algebraic homomorphism $\chi_v:T_v\to F_v^*,t_v\mapsto \chi_v(t_v)$ and a continuous homomorphism $|\chi|_F:T(\bba_F) \to V_F\subseteq \bbr_+$
defined by
\[
(t_v)_v\mapsto \prod_v|\chi_v(t_v)|_v.
\]

\begin{defn}\label{def:GlobUnrChar}
Given $T$ defined over $F$ and splitting over $E/F$, let $T(\bba_F)^1 = \cap_\chi \ker|\chi|_F$
with $\chi$ running over $X^\ga$. An \emph{unramified (automorphic) character} of $T$ is a continuous homomorphism $T(\bba_F)/T(F)\to \bbc^*$ trivial on $T(\bba_F)^1$. The group of all such characters is denoted $X_{T}$.
 \end{defn}

It is well known \cite[Section I.1.4]{MW} that the group \(X_T\) of unramified characters is a connected abelian Lie group with a complex analytic structure. We will recall this in more detail in the next subsection.

\subsection{Compactness and identity components} In this section, we shall study some topological features of the space $\Hom_c(\Hom_{\Gamma_\ef}(X,C_E),\bbc^*)$.  
In light of Langlands's result, Theorem \ref{thm:LangTori}, the information extracted in this section will provide a proof of Theorem \ref{thm:mainThm0}. To ease notations, let us make the convention
\begin{equation}\label{eq:TCF}
T_{C_F} := \Hom_{\Gamma_\ef}(X,C_E).    
\end{equation}
We assume here that the field $F$ is either local or global with finite Galois extension $E/F$ splitting $T$. The group $T_{C_F}$ is a locally compact topological group with topology determined by that of $C_E$. 
It is known that when \(F\) is local we have an identification 
\(T_{C_F}= T(F)\) while for \(F\) global we have a closed inclusion with finite-index \(T(\bba_F)/T(F) \subset T_{C_F}\) (see \cite[p. 245]{La}). 

It is convenient to extend the notion of the group of unramified characters to $T_{C_F}$ as well.
If $\chi\in X^{\Gamma_\ef}$ is a rational character of $T$, then for any $t\in T_{C_F}$, we have $\chi(t)\in C_F$ and hence $|\chi|_F = |\cdot|_F\circ \chi$ is a continuous character of $T_{C_F}$. Similarly to the definitions in the local and global cases, let us write
\begin{equation}\label{eq:TCF-one}
   T_{C_F}^1 = \cap_{\chi\in X^{\Gamma_\ef}} \ker(|\chi|_F). 
\end{equation}
Finally, we say that a continuous character $\phi$ of $T_{C_F}$ is \emph{unramified} if $\phi$ vanishes on $T_{C_F}^1$, and we denote by  
\[\widetilde X_{T}:=\Hom_c(T_{C_F}/T_{C_F}^1,\bbc^*)\] the 
group of unramified characters of $T_{C_F}$.
\begin{prop}\label{prop:MaxCompact}
    For $F$ local or global, 
    we have
    that $T_{C_F}^1$ is the maximal compact subgroup of $T_{C_F}$.
\end{prop}

\begin{proof}
    We claim that $T_{C_F}^1 = \Hom_{\Gamma_\ef}(X,C^1_E)$. Indeed, note that by \eqref{eq:TCF-one}, we have $\Hom_{\Gamma_\ef}(X,C^1_E)\subseteq T_{C_F}^1$.
    Conversely, let $t\notin \Hom_{\Gamma_\ef}(X,C^1_E)$, that is, $t\in T_{C_F}$ and 
    there exists $x\in X$ with $|x(t)|_E>1$. We show that $t\not\in T^1_{C_F}$. Indeed, for  $x\in X$ as above, let 
    \[
    \tilde{x} = \sum_{\gamma \in \Gamma_{\ef}} \gamma(x)\in X^{\Gamma_{\ef}}
    \]
    and note that $\tilde x(t) =\prod_{\gamma \in \Gamma_{\ef}}\gamma(x(t))=N_{E/F}(x(t))\in C_F,$
    where $N_{E/F}$ is the norm map. It follows that $|\tilde x(t)|_F = |N_{E/F}(x(t))|_F=|x(t)|_E>1,$ from which we conclude that $t\not\in T_{C_F}^1$. As $C_E^1$ is compact, then $T_{C_F}^1$ is compact. Now, $T^1_{C_F}$ contains every compact subgroup of $T_{C_F}$, since $\bbr_+$ does not admit any non-trivial compact subgroups. This finishes the proof.
\end{proof}

\begin{cor}\label{cor:compactness}
    If $F$ is a global field, 
    then $T(\bba_F)^1/T(F)$ is compact.
    If $F$ is a local field, then $T(F)^1$ is the maximal compact subgroup of $T(F)$. 
\end{cor}

\begin{proof}
    In the global case we have $T(\bba_F)^1/T(F) = T^1_{C_F}\cap (T(\bba_F)/T(F))$, and since $T^1_{C_F}$ is compact and $T(\bba_F)/T(F) \subset T_{C_F}$ is closed, then  $T^1(A_F)/T(F)$ is compact. If $F$ local, $T_{C_F} = T(F)$ so the claim follows from Proposition \ref{prop:MaxCompact}.
\end{proof}

We are interested in describing the quotient  \(T_{C_F}/T_{C_F}^1\) and the space \(\Hom_c(T_{C_F}/T_{C_F}^1,\bbc^*)\). Consider the continuous restriction map $R:\Hom_c(T_{C_F},\bbc^*)\to \Hom_c(T_{C_F}^1,\bbc^*)$. Its kernel consists of those characters that are trivial on $T_{C_F}^1$, hence is naturally identified with $\Hom_c(T_{C_F}/T_{C_F}^1,\bbc^*)=\widetilde X_{T}$.
To further analyse this structure, we now introduce a logarithmic map, similarly to what was done in \cite[Section I.1.4]{MW}.
Given the lattice of rational characters $X^\ga$ of $T$, consider the real vector spaces $\fa_0^* = \bbr \otimes X^\ga$ and $\fa_0=\Hom(X^\ga,\bbr)$. Define a homomorphism of abelian groups $\logtc:T_{C_F}\to \fa_0$ characterized by
\begin{equation}\label{eq:LogMap}
\lpi \chi,\logtc(t) \rpi = \log_q(|\chi|_F(t)),
\end{equation}
for $\chi\in X^\ga$ and $t\in T_{C_F}$. Here, $q = e$ if $V_F\cong\bbr$ while, if $V_F\cong\mathbb{Z}$, then $q=q_F$ is the base of the valuation group $V_F$.
\begin{prop}\label{prop:tildeXT}
Let $F$ be a local or global field.
Let $\widetilde{L}_T:=\logtc(T_{C_F})$. Then $\widetilde{L}_T\approx T_{C_F}/T_{C_F}^1$.  
In the \(\bbr\)-cases we have $\widetilde{L}_T=\fa_0$ while in the \(\bbz\)-cases $\widetilde{L}_T \subset \fa_0$ is a sublattice of $\Hom(X^\ga,\bbz)$ of finite index.
Furthermore, there is a continuous and open surjection
$\fe:\Lie(\dT)^\ga\to \widetilde{X}_T$ which is an isomorphism in the \(\bbr\)-cases and its kernel is $\Hom(\widetilde{L}_T,\bbz)$ in the \(\bbz\)-cases. In particular, in all cases, $\widetilde{X}_T$ is a connected abelian complex linear algebraic group whose canonical complex structure is compatible with that of $\Lie(\dT)^\ga$. 
\end{prop}
\begin{proof}
In both the $\bbr$- and $\bbz$-cases, we have $\ker(\logtc)=T_{C_F}^1$. By Proposition \ref{prop:MaxCompact} the map $\logtc$ is proper, and hence
defines a topological isomorphism from $T_{C_F}/T_{C_F}^1$ onto $\widetilde{L}_T$. 
Thus, we have a topological identification $\widetilde{X}_T \approx \Hom_c(\widetilde{L}_T,\bbc^*)$. Observe that from this identification, then $\widetilde{X}_T$ carries a canonical complex structure as elucidated in \cite[Sections 2.4 and 4.4]{Ta1} in the $GL_1$-case (see also \cite[Section I.1.4]{MW}).

The characterization of \(\widetilde{L}_T\approx T_{C_F}/T_{C_F}^1\) as a vector space (in the $\bbr$-cases), or a lattice (in the $\bbz$-cases), follows from \eqref{eq:LogMap}, similarly as was done in \cite[pp. 6,7]{MW}.

If $V_F\cong\bbr$, using $\widetilde{L}_T = \fa_0\cong \bbr\otimes \Hom(X^\ga,\bbz)$ and the tensor-hom adjunction we obtain
\begin{equation}\label{eq:Rcancomplexstruc}
\widetilde{X}_T \cong \Hom_c(\widetilde{L}_T,\bbc^*)
\cong\Hom_c(\bbr, \Hom(\Hom(X^\ga,\bbz),\bbc^*))
\cong \Lie(\dT)^\ga,
\end{equation}
where we used that any continuous homomorphism between finite-dimensional Lie groups is necessarily smooth  (see \cite[Theorem 3.39]{Wa}). Furthermore, denoting by $\fa^*$ the complexification of $\fa_0^*$, note that $\Lie(\dT)^\ga \cong \fa^*$.

In the \(\bbz\)-cases, we have that $\widetilde{L}_T \subseteq \Hom(X^\ga,\bbz)$ 
is a sublattice (by the same argument as for the analogous 
case of $T(\bba_F)$, see \cite[pp. 6,7]{MW}). In particular 
this inclusion is of finite index, and we have $\Hom(\widetilde{L}_T,\bbc) \cong \bbc\otimes \Hom(\widetilde{L}_T,\bbz) \cong \Lie(\dT)^\ga$.
As $\widetilde{L}_T$ is free (and discrete), applying $\Hom(\widetilde{L}_T,-)$ to the exponential exact sequence yields 
\begin{equation}\label{eq:Zcancomplexstruc}
    0 \to \Hom(\widetilde{L}_T,\bbz) \to \Lie(\dT)^\ga \to \widetilde{X}_T \to 0.
\end{equation}
Both in \eqref{eq:Rcancomplexstruc} and \eqref{eq:Zcancomplexstruc}, we see that the canonical complex structure on $\widetilde{X}_T$ is compatible with the complex structure on $\Lie(\dT)^\ga$, finishing the proof. 
\end{proof}

When \(F\) is global, by using the closed inclusion \(T(\bba_F)/T(F) \subset T_{C_F}\), it is clear that the restriction of \(\logtc\) to \(T(\bba_F)/T(F)\) coincides with the map \(\log_T:T(\bba_F)/T(F)\to\mathfrak{a}_0\) defined in \cite[p. 7]{MW} (their map is defined from \(T(\bba_F)\to\mathfrak{a}_0\), but it clearly factors through \(T(\bba_F)/T(F)\to\mathfrak{a}_0\)), whose kernel is \(T(\bba_F)^1/T(F)\). We conclude that there is a topological identification of \(T(\bba)/T(\bba_F)^1\) with \(L_T\), the image of \(\log_T\). When \(F\) is local, let us write \(L_T := \widetilde{L}_T\).  We have the following characterization of the space \(X_T\) of unramified characters.

\begin{cor}\label{cor:GlobXT} 
    Let \(F\) be a local or global field. There is a continuous and open surjection
$\fe:\Lie(\dT)^\ga\to X_T$ whose  kernel is the lattice $\widehat{L}_T:=\Hom(L_T,\bbz)\subset \fa_0^*$ in the \(\bbz\)-cases, and which is an isomorphism in the \(\bbr\)-cases. In particular, $X_T$ is always a connected abelian complex linear algebraic group whose canonical complex structure is compatible with that of $\Lie(\dT)^\ga$. 
\end{cor}

\begin{proof}
    Follows from Proposition \ref{prop:tildeXT}.
\end{proof}

\begin{prop}\label{prop:charTCFconcomp}
    Let $F$ be a local or global field. 
    Then $\widetilde X_{T}\approx \Hom_c(T_{C_F}/T_{C_F}^1,\bbc^*)=\Hom_c(T_{C_F},\bbc^*)^\circ$. Furthermore, we have that $\widetilde X_{T}$ is an open subgroup of $\Hom_c(T_{C_F},\bbc^*)$, and $\Hom_c(T_{C_F},\bbc^*)/\widetilde X_{T}\approx 
    \Hom_c(T^1_{C_F},\bbc^*)$. 
\end{prop}

\begin{proof}
    As $T^1_{C_F}$ is compact, we have $\Hom_c(T^1_{C_F},\bbc^*)=\Hom_c(T^1_{C_F},\bbt)$ is discrete, as it is well known that the Pontryagin dual of a compact topological group is discrete \cite[(23.17) Theorem]{HR}. Hence, the kernel of $R$, $\widetilde X_T$, is a subgroup which is open and closed, and then must be a connected component, provided it is connected. 
    From \cite[Theorem, p.158]{Ma}, it suffices to show that $T_{C_F}/T_{C_F}^1$ is torsion-free, which in turn 
    follows immediately from Proposition \ref{prop:tildeXT}. 
    
    Finally, to prove the statement about the component group, we need to show that the restriction homomorphism $R:\Hom_c(T_{C_F},\bbc^*)\to \Hom_c(T_{C_F}^1,\bbc^*)$ is surjective. If $V_F\cong\bbz$ then $T^1_{C_F}\subset T_{C_F}$ is open, hence any algebraic extension of a continuous complex character of $T^1_{C_F}$ to $T_{C_F}$  (which exists since $\bbc^*$ is divisible) is continuous. For $V_F\cong\bbr$, we have a splitting $C_E=C_E^1\times V_E$ as $\Gamma_\ef$-modules (see Remark \ref{rem:GalSplitting}), hence a direct product decomposition $T_{C_F}=T^1_{C_F}\times T_\bbr$, where $T_\bbr$ is a real split torus. Again, we see that $R$ is surjective.
\end{proof}

\begin{prop}\label{prop:OpClosed}
    Let $F$ be a global field. Then, the restriction map
    \[
    \rho:\Hom_c(T_{C_F},\bbc^*) \to \Hom_c(T(\bba_F)/T(F),\bbc^*)
    \]
    is continuous, open, and closed if both spaces are endowed with the compact-open topology. Moreover, 
    $\rho$ is surjective with finite kernel.
\end{prop}

\begin{proof}
    When $F$ is global, we have that $T(\bba_F)/T(F) \subseteq T_{C_F}$ is a closed inclusion of finite index \cite[p. 245]{La}. Hence, $T(\bba_F)/T(F)$ is also an open subgroup of $T_{C_F}$ and thus $\rho$ is a continuous and open surjection (see \cite[(24.5) Theorem]{HR} for the $\bbt$-valued case; their proof that this homomorphism is also open can be easily adapted to the $\bbc^*$-valued case). As $\ker(\rho)$ is finite, this map is also closed.
\end{proof}

\begin{cor}\label{cor:main0}
If $F$ is a local or global field, then $\mathsf{X}_\bbc(T,F)$ is an abelian Lie group with complex analytic structure, and identity component
$X_T$. Furthermore, the component group $\mathsf{X}_\bbc(T,F)/X_T\approx \Hom_c(T(\bba_F)^1/T(F),\bbc^*)$.

\end{cor}
\begin{proof}
Obviously, $\mathsf{X}_\bbc(T,F)$ is an abelian topological group. To show that it can be given a compatible structure of a complex analytic Lie group, it is enough to show this for an open subgroup. Hence, it suffices to show that the subgroup $X_T\subset \mathsf{X}_\bbc(T,F)$ is open, since $X_T$ is a connected complex analytic Lie group, by Corollary \ref{cor:GlobXT}. The claim that $X_T$ is the identity component then simply follows, $X_T$ being closed, open and connected. 

    If $F$ local, we have $T_{C_F} = T(F)$, and hence $\widetilde{X}_T = X_T$. Now Proposition \ref{prop:charTCFconcomp} shows that $X_T$ is an open subgroup. 
    If $F$ is global, then Propositions \ref{prop:charTCFconcomp} and \ref{prop:OpClosed} show that $X_T=\rho(\widetilde X_{T})$ is open. 
    
    Finally, the proof of the claim on the component group is similar to Corollary \ref{prop:charTCFconcomp}.
\end{proof}

With the description of $T(F)^1$ at hand, we can be more precise in the description of $X_T$, in the local case. Given a subtorus $S$ of $T$ we let $X(S)$ denote its character lattice.

\begin{prop}\label{prop:locT1} Let $F$ be a local field.
    The following assertions hold true:
    \begin{itemize}
        \item[(a)] Let $S\subset T$ be the largest $F$-split subtorus of $T$.
        The natural homomorphism $S(F) \to T(F)/T(F)^1$ factors through a homomorphism 
        $S(F)/S(F)^1 \to T(F)/T(F)^1$ which is injective with closed image and finite index. 
        \item[(b)] If $F$ is a local Archimedean field, then $X_T \approx \Lie(\dT)^{\Gamma_\ef}$. If $F$ is a local non-Archimedean field, then there is a finite-to-one map $X_T \to X_S \approx \dT_{\Gamma_\ef}$.
    \end{itemize}
\end{prop}

\begin{proof}
We start by proving (a). The composition $S(F) \to T(F) \to T(F)/T(F)^1$ descends to an injective homomorphism
    \[
    \frac{S(F)}{S(F)\cap T(F)^1} \to \frac{T(F)}{T(F)^1}.
    \]
    On the one hand, from the finite-index injection $X^{\Gamma_\ef}\to X(S)$ (see \cite[p. 6]{MW}), we have $S(F)^1 \subseteq S(F)\cap T(F)^1$. On the other hand, as $S(F)$ is closed in $T(F)$ and $T(F)^1$ is compact, we obtain that $S(F)\cap T(F)^1\subseteq S(F)^1$, by maximality of the latter (see Corollary \ref{cor:compactness}), and hence we have an injective map 
    \begin{equation}\label{eq:TSonemap}
    \frac{S(F)}{S(F)^1} \to \frac{T(F)}{T(F)^1}.
    \end{equation}
    We now show that the image of \eqref{eq:TSonemap} has finite index. We first describe the cokernel of \eqref{eq:TSonemap} as a quotient of $T(F)/S(F)$. Indeed, note that the image of \eqref{eq:TSonemap} equals $(S(F)\cdot T(F)^1)/T(F)^1$, hence the cokernel of \eqref{eq:TSonemap} is isomorphic to $T(F)/(S(F)\cdot T(F)^1)$. 
    
    Next observe that by Hilbert 90, the group $T(F)/S(F)$ identifies with $\overline{T}(F)$, where $\overline{T} = T/S$ as algebraic groups. Using this, we can rewrite the cokernel of \eqref{eq:TSonemap} as:
    $T(F)/(S(F)\cdot T(F)^1)\approx \overline{T}(F)/\overline{T(F)^1}$, 
    where we wrote $\overline{T(F)^1}$ for the 
    natural image of $T(F)^1$ in $\overline{T}(F)$. 

    Now let $T_a$ denote the maximal anisotropic subtorus of $T$, then by the almost direct product decomposition $T=T_a\cdot S$ (see \cite[Proposition 8.15]{Bo}) as algebraic tori, there is a natural epimorphism
    of algebraic tori $T_a \to \overline{T}$. 
    Recall that $T_a(F)$ is compact (see \cite[Section 35.3]{Hu} for the Archimedean case and \cite{Pr} for the non-Archimedean case) and hence contained in $T(F)^1$ by Corollary \ref{cor:compactness}. Therefore, denoting by $\overline{T_a(F)}$ the image of $T_a(F)$ in $\overline{T}(F)$, we obtain:
    \begin{equation}
    T(F)/(S(F)\cdot T(F)^1)\approx\overline{T}(F)/\overline{T(F)^1}
    \approx(\overline{T}(F)/\overline{T_a(F)})/(\overline{T(F)^1}/\overline{T_a(F)}).
    \end{equation}
     By (the proof of) \cite[Lemma 3.7]{Con}, this is finite, as was to be shown.
    As for item (b), when $V_F\cong \bbr$, \eqref{eq:TSonemap} is a finite-index inclusion of vector spaces, thus an isomorphism. Hence, $X_T \approx \Lie(\dT)^{\Gamma_{\ef}}$. If $F$ is local non-Archimedean, then as $S(F)/S(F)^1\approx  \dX^{\Gamma_\ef}$ we have $X_S \approx \Hom_c(S(F)/S(F)^1,\bbc^*)\approx \dT_{\Gamma_{\ef}}$  and \eqref{eq:TSonemap} implies the map $X_T\to X_S$ is finite-to-one.
\end{proof}

\begin{rem}
    When $F$ is local non-Archimedean, a precise description of $X_T$ was given in \cite[Section 3.3.1]{Ha}. We recover this description below, in Proposition \ref{prop:LnAHaines}. If the extension $E/F$ is unramified, then $X_T\approx X_S$
    (see \cite[Section 9.5]{Bo2}). 
\end{rem}

\section{Cohomological side}\label{s:CohoSide}
In order to discuss the continuity of the Langlands's map, we obviously need to topologize the cohomology groups first. We shall use the theory developed by C. Moore in the series of papers \cite{Mo1,Mo2,Mo3,Mo4}. In fact, we shall closely follow the axiomatic treatment of \cite[Definition, p. 16]{Mo3}. See also \cite{AM} and \cite[Section 3]{Ra} for other accounts on this theory.

\subsection{Measurable cohomology theory} We briefly recall Moore's construction of his cohomology theory in the next paragraph, for any locally compact group $G$ and Polish $G$-modules $A$ (i.e., Polish abelian groups -- abelian topological groups whose topology admits a separable complete metric \cite[Proposition 1]{Mo3} -- endowed with a continuous $G$-action). We are interested in $A$ a lattice, a vector space or a complex torus, so this theory is well suited. 

First, given any $\sigma$-finite measurable space $(X,\mathcal{B},\mu)$ with $(X,\mathcal{B})$ countably generated and  $A$ a separable metrizable space, choose a finite measure $\nu$ equivalent to $\mu$, and a metric $\rho$ on $A$ of finite diameter. Define $I(X,A)$ as the set of equivalence classes of $\nu$-measurable functions $X \to A$ under the almost-everywhere-equality equivalence relation and equipped with the metric
\begin{equation}\label{eq:metric}
\bar\rho(\varphi,\psi) = \int_{X} \rho(\varphi(x),\psi(x)) d\nu(x).
\end{equation}
The topology on $I(X,A)$ defined by $\bar\rho$ depends only on the measure class of $\nu$ and the topology of $A$ (see \cite[Corollary to Proposition 6]{Mo3}) and not on the particular choices made. Now, let $G$ be a locally compact group with Haar measure $\mu$ and $A$ a Polish $G$-module as above. Define the chain complex $(\uC^*(G,A),\underline{\delta}^*)$ where
$\uC^n(G,A) = I(G^n,A)$ endowed with the metric $\bar\rho$ of (\ref{eq:metric}) and 
$\underline{\delta}^n$ is the usual algebraic coboundary operator restricted to elements of $\uC^n(G,A)$. The subgroups $\uZ^n(G,A),\uB^n(G,A)$ are given the induced topology. We then let $\uH^n(G,A)$ denote the cohomology groups of the chain complex $(\uC^*(G,A),\underline{\delta}^*)$, topologized with the quotient topology. The following proposition summarizes what we need from this construction. All proofs are found in \cite{Mo3} and \cite{AM}.

\begin{prop}\label{prop:MooreSummary}
The following assertions hold true:
\begin{itemize}
\item[(a)] The coboundary operators $\underline{\delta}^n$ are continuous for all $n$.
\item[(b)] If $\uB^n(G,A)$ is a closed subgroup of $\uZ^n(G,A)$, then $\uH^n(G,A)$ is a Polish abelian group.
\item[(c)] For any locally compact $G$ and Polish $A$ we have $\uH^1(G,A) \cong H^1_c(G,A)$ as abelian groups, which is a homeomorphism if $H^1_c(G,A)$ is endowed with the quotient topology and $Z^1_c(G,A),B^1_c(G,A)$ are given the compact-open topology.
\item[(d)] For any short exact sequence $0 \to A \stackrel{i}{\to} B \stackrel{p}{\to} C \to 0$ of Polish $G$-modules with $i$ a homeomorphism onto its image and $p$ continuous and open, all morphisms in the long exact sequence induced by $\uH^*(G,-)$ are continuous.
\item[(e)] Given any $f:G\to G'$ continuous homomorphism between locally compact groups, $A,A'$ Polish $G$- and $G'$-modules and $\varphi:A'\to A$ a continuous homomorphism of abelian groups satisfying
\[
\varphi(f(g)\cdot a') = g\cdot \varphi(a'),
\]
then the induced morphism $\uH^n(G',A') \to \uH^n(G,A)$ is continuous.
\item[(f)] Suppose that the quotient $G/G^\circ$ is compact.
If $A$ is discrete, then $\uH^p(G, A)$ is countable and discrete in its quotient topology, for all $p>0$. If $A$ is a finite-dimensional real vector space, then $\uH^p(G, A)$ is a finite-dimensional real vector space in its quotient topology, for all $p>0$. 
\item[(g)] If $G$ is compact and $A$ is a finite-dimensional real vector space, then $\uH^p(G, A)=0$, for all $p>0$. 
\end{itemize}
\end{prop}

\begin{proof}
Item (a) is \cite[Proposition 20]{Mo3}. For item (b), see the discussion on \cite[p. 10]{Mo3}. 
Item (c) is \cite[Theorem 3, Corollary 1, Theorem 7]{Mo3}. Item (d) is \cite[Proposition 25]{Mo3} and item (e) is, with minor modifications, \cite[Proposition 27]{Mo3}. Assertion (f) is in \cite[Theorem D]{AM} and (g) is in \cite[Theorem A]{AM}.
\end{proof}

Note that item (e) of the previous proposition implies that all homomorphisms between cohomology groups in (Moore's version of) the Lyndon-Hochschild-Serre spectral sequence (see \cite[Chapter I, Sections 3, 4 and 5]{Mo1} and \cite[Theorem 9]{Mo3}) are continuous. In particular, this applies to the maps in the inflation-restriction exact sequence. We shall make extensive use of this property.

\subsection{Low-dimensional measurable cohomology groups}
The biggest issue when dealing with the theory of measurable cohomology groups as described by Moore is the statement (b) in Proposition \ref{prop:MooreSummary}: In general, the measurable groups $\uH^r(G, A)$ may not be Hausdorff (for an explicit example, see \cite[Lemma 7.7]{AM}). 
There are several partial results to describe when these groups are Hausdorff, and for a discussion on this topic, we refer to \cite[Section 1.4]{AM}. 

Throughout this section, $T$ is an algebraic torus over 
$F$, and $E/F$ is a finite Galois extension of $F$ over which 
$T$ splits. Our first
task will be to prove that 
in our context, the cohomology groups $H^1_c(\we, A)$ are complete metric spaces when we take $A$ to be $\dT$, $\Lie(\dT)$ or $X$. We note that in some possibilities for $W_\ef$, this is a consequence of \cite[Theorem D]{AM}, but we will give an alternative proof that works for all possible cases of $W_\ef$.  Note also that when $T$ is split over $F$, the relevant first group of continuous cohomology becomes the space of all continuous functions from $C_F$ to $A$, which is naturally a metrizable space, when endowed with the compact-open topology. Before we continue, we need a preliminary result.

\begin{lem}\label{lem:coinvTrick}
Let $\Gamma$ be a finite group acting linearly on a finite-dimensional real vector space $U$ with $U^\Gamma = 0$.
Suppose that $(u_m)$ is a sequence in $U$ such that $(\gamma(u_m) - u_m)$ converges for all $\gamma\in\Gamma$. Then, $(u_m)$ converges.
\end{lem}

\begin{proof}
Since $U$ is completely reducible, $U^\Gamma=0$ 
implies $U_\Gamma = 0$, so that $U = \sum_{\gamma\in\Gamma}(\gamma - \id)(U)$. 
Thus, any element $u\in U$ can be written as a summation $u = \sum_\gamma (\gamma^{-1} - \id)(u_\gamma)$ for suitable elements $u_\gamma\in U$. Now let $(-,-)$ be a $\Gamma$-invariant inner product on $U$. Then, 
\[
(u_m,u) = \sum_\gamma(u_m, (\gamma^{-1} - \id)(u_\gamma)) = \sum_\gamma(\gamma(u_m) - u_m,u_\gamma)
\]
which, from the assumptions, imply $(u_m,u)$ converges and so does $(u_m)$.
\end{proof}

\begin{thm}\label{thm:Hausdorff}
For $F$ local or global, we have that $H^1_c(\we,\dT)$ is a Hausdorff\footnote{In fact, they are Polish abelian groups, by Proposition \ref{prop:MooreSummary}(b) and (c).} abelian group.
\end{thm}

\begin{proof}
It suffices to show that $\uB^1(\we,\dT)$ is closed in $\uC^1(\we,\dT)$. To that end, let $(b_m)_{m\in\bbn}$ be a sequence in $\uB^1(\we,\dT)$ converging to some $f \in \uC^1(\we,\dT)$. Since $ \uC^1(\we,\dT) $ is a metric space (with metric as in \eqref{eq:metric}), it suffices to check sequential closedness, and we shall show that $f \in \uB^1(\we,\dT)$. Associated to the sequence $(b_m)$ is a sequence $(t_m)$ in $\dT$ such that $b_m(\omega) = \omega(t_m)t_m^{-1}$ for each $\omega\in \we$. Furthermore, for each $m\in\bbn$ we can choose $\nu_m\in \Lie(\dT)$ such that $\exp(\nu_m) = t_m$. Then, for all $\omega\in\we$
\[
b_m(\omega) = \omega(\exp(\nu_m))\exp(-\nu_m) = \exp(\omega(\nu_m)-\nu_m),
\]
since the action of $\we$ commutes with the exponential map. We will show that we can modify the sequence $(\nu_m)$ without changing $(b_m)$ in such a way that $(\nu_m)$ becomes convergent, and hence the limiting cocycle is in fact a coboundary.

From \cite[Proposition 6]{Mo3}, there exists a measure-zero set $S\subseteq\we$ outside which $b_m(\omega) \to f(\omega)$ pointwise. Hence, we can choose representatives $\{\gamma\} \subseteq \we\setminus S$ for the quotient $\we/C_E \cong \ga$ such that, for each $\gamma$, we have
 \[
b_m(\gamma)=\exp(\gamma(\nu_m)-\nu_m)\to f(\gamma).
\]
Write $\nu_m = \lambda_m + i \mu_m$, where $\lambda_m$ and $\mu_m$ are in the real form $\Lie(\dT)_0 =  \bbr\otimes X$ of $\Lie(\dT)=\bbc\otimes X$. First, note that each $\nu_m\in \Lie(\dT)$ is determined up to translations by $(2\pi i)X$. So, we can and will assume that the imaginary parts $(\mu_m)$ lie in a bounded region of $\Lie(\dT)_0$. Passing to a subsequence, if needed, we might as well assume $(\mu_m)$ is convergent. Then, since $\lambda_m = \nu_m - i\mu_m$, we get
\[
\exp(\gamma(\lambda_m)-\lambda_m) =b_m(\gamma)\exp(\gamma(i\mu_m)-i\mu_m)^{-1}
\]
and hence $\exp(\gamma(\lambda_m)-\lambda_m)$ converges, since both factors in the right-hand side converge. Using that the exponential map restricted to $\Lie(\dT)_0$ is a homeomorphism onto its image, it then follows that
$(\gamma(\lambda_m)-\lambda_m)$ is convergent. Now, if $\nu \in \Lie(\dT)_0^\ga$, then 
\[(\gamma(\lambda_m + \nu)-(\lambda_m+\nu))=(\gamma(\lambda_m)-\lambda_m).\]
Hence, denoting by $\textup{Proj}_U$ the projection (along $\Lie(\dT)_0^\ga$) onto a $\ga$-stable complement $U$ of $\Lie(\dT)_0^\ga$, we may replace $(\lambda_m)$ by $(\textup{Proj}_U(\lambda_m))$ without changing the initial coboundary sequence $(b_m)$.
 We can now use Lemma \ref{lem:coinvTrick} to conclude that $(\lambda_m)$ converges and, hence, $\lim_m\nu_m = \nu$ for some $\nu\in\Lie(\dT)$. Setting $t = \exp(\nu)$ we obtain that $f(\gamma) = \gamma(t)t^{-1}$ for all $\gamma$. Hence, as $\we$ can be written as the disjoint union $\we = \cup_\gamma \gamma C_E$ and $C_E$ acts trivially on $\dT$, we get $b_m(\omega)\to f(\omega)=\omega(t)t^{-1}$ for all $\omega\notin S$, i.e., $f\in\uB^1(\we,\dT)$.
\end{proof}

In fact, the proof 
of Theorem \ref{thm:Hausdorff} yields  similar consequences for $X$ and $\Lie(\dT)$.

\begin{cor}\label{cor:Hausdorff}
    For $F$ local or global, both $H_c^1(W_\ef,X)$ and $H_c^1(W_\ef,\Lie(\dT))$ are Hausdorff\footnote{Similar comment to the one of Theorem \ref{thm:Hausdorff} applies.} abelian groups.
\end{cor}

\begin{proof}
As before, from Proposition \ref{prop:MooreSummary}, items (b) and (c), it suffices to show that the respective measurable coboundary group is closed in the cases $A=X$ or $A=\Lie(\dT)$ as well. The proof in the case $A=\Lie(\dT)$ is analogous but easier than the case $A=\dT$ treated in Theorem \ref{thm:Hausdorff}.  The case $A=X$ is even easier: 
Given $(b_m)$ a sequence in $\uB^1(W_\ef,X)$ that converges to $f\in \uC^1(W_\ef,X)$, we can choose an associated sequence $(x_m)$ in $X$ such that $b_m(\omega) = \omega(x_m) - x_m$
for all $\omega\in W_{E/F}$. This implies that for all $m$, the 
coboundary $b_m:W_{E/F}\to X$ is constant on $C_E$-cosets, hence 
descends to a coboundary $b^0_m:\Gamma_{E/F}\to X$ on the finite quotient $\Gamma_{E/F}$ of $W_{E/F}$. 
As in the proof of Theorem \ref{thm:Hausdorff}, $b^0_m:\Gamma_{E/F}\to X$ must converge pointwise, which is  equivalent to $b^0_m$ being stationary for sufficiently large $m$. This implies that 
there exists $x\in X$ such that $b^0_m$ converges to the
coboundary $b^0(\gamma):=\gamma(x)-x$, and hence that 
$f(\omega)=\omega(x)-x$ is a coboundary.
\end{proof}

\begin{prop}\label{prop:DiscreteIsDiscrete}
    Let $F$ be local or global. Then, $\uH^1(W_\ef,X)$ and $\uH^2(W_\ef, X)$ are countable and, as topological spaces, discrete.
\end{prop}

\begin{proof}
It is a general fact that a countable quotient 
$B/C$, where $B$ is a Polish group with normal subgroup $C$ which is an analytic subset, is necessarily discrete as a topological space, and hence Hausdorff. For this statement and its proof, we refer the reader to the proof of Proposition 7.2 in \cite{AM}.  
Hence, it suffices to show that $\uH^r(W_\ef, X)$ is 
countable for $r = 1, 2$.     

    In the $\bbr$-case, note that $C_E$ is almost connected in the sense that $C_E/C_E^\circ$ is compact; it is profinite in the global case \cite[Chapter VIII, \textsection 2]{NSW} and finite in the local case.
    Hence, $W_\ef$ is almost connected, and it follows from \cite[Proposition 1.3]{Mo2} that $\uH^r(W_\ef,X)$ is countable for $r=1,2$. 
    
    For the $\bbz$-case, we use the computations done in \cite[Proposition 6]{Ra} to deduce that 
    $\uH^r(W_\ef,X)$ is countable, for $r=1,2$. 
    Indeed, first recall that as $\bbz$ is free and discrete, then the measurable, continuous and abstract cohomology theories agree (see table in \cite[p. 913]{AM}). We thus have $\uH^r(\bbz, A) = 0$ for any $\bbz$-module $A$ and for all $r>1$ \cite[Remark, p.170]{Wb}.
    So, from the exact sequence $0\to W^1_F \to W_F \to \bbz\to 0$, where $W^1_F$ is compact (and profinite) and $W_F$ is the absolute Weil group of $F$,
    we obtain from the inflation-restriction exact sequence that
    \[
    0\to \uH^1(\bbz,X^{W^1_F})\to\uH^1(W_F, X)\to \uH^1(W^1_F,X)^{\lpi\textup{Fr}\rpi}\to 0.
    \]
    For
    any $W_F$-module $A$, we have $\uH^1(\bbz, A) = A/(\textup{Fr}-1)A$ with $\textup{Fr}$ the image of $1\in \bbz$ in $\Aut(A)$, so $\uH^1(\bbz,X^{W^1_F})$ is countable. As $W^1_F$ is profinite, it is almost connected, so we also have that $\uH^r(W_F^1,X)^{\lpi\textup{Fr}\rpi}$ is
    countable for $r=1,2$ \cite[Proposition 1.3]{Mo2}, from which $\uH^1(W_F, X)$ is countable. But using again that $\uH^r(\bbz, A) = 0$ for $r>1$, we get the higher inflation-restriction sequence just as in the proof of \cite[Proposition 6]{Ra} for $n=2$:
    \[
    0\to \uH^1(\bbz,\uH^{1}(W^1_F, X))\to\uH^2(W_F, X)\to \uH^2(W^1_F,X)^\textup{Fr}\to 0.
    \]
    We conclude, just as for $\uH^1(W_F, X)$, that $\uH^2(W_F, X)$ is countable. From the extension $0\to W_E^c \to W_F \to W_\ef \to 0$,  the inflation-restriction exact sequence yields
    \begin{multline*}
        0 \to \uH^1(W_\ef, X) \to \uH^1(W_F, X) \to \\
        \to \uH^1(W_E^c, X)^{W_\ef}\to\uH^2(W_\ef, X) \to \uH^2(W_F, X).
    \end{multline*}
    Recall that $W_E^c$ is compact.
    Since $W_E$ acts trivially on $X$ then $\uH^1(W_E^c,X)\cong \Hom_c(W_E^c,X)=0$ by the compactness of $W_E^c$, and
    we conclude that $\uH^r(W_\ef, X)$ for $r=1,2$ is also countable, finishing the proof.
\end{proof}

\begin{prop}\label{prop:H1middle}
For $F$ a global or a local field,
we have a $\bbc$-linear isomorphism
\[
\uH^1(\we,\Lie(\dT)) \approx \Lie(\dT)^\ga.
\]
\end{prop}
\begin{proof}
Consider the exact sequence $0\to W_\ef^1 \to \we \to V_F \to 0$. As $W_\ef^1$ is compact and $\Lie(\dT)$ is a finite-dimensional real vector space, we obtain from Proposition \ref{prop:MooreSummary}(g) that $\uH^1(W_\ef^1,\Lie(\dT))=0$. 
Recall that the action of $W_{E/F}$ on $\Lie(\dT)$ factors 
through a $\bbc$-linear action of the finite group $\Gamma_{E/F}$.
In the $\bbr$-cases, the canonical action of 
$V_F$ on $\Lie(\dT)^{W_{E/F}^1}=\Lie(\dT)^{\Gamma_\ef}$ is trivial, so from the inflation-restriction exact sequence, 
we obtain a continuous isomorphism of abelian Lie groups $\Hom_c(V_F,\Lie(\dT)^{\Gamma_\ef})\cong\uH^1(\we,\Lie(\dT))$. This map is explicitly defined (on the level of cocycles) by
\[
\varphi \mapsto \zeta_\varphi:\omega \mapsto \log(|\omega|_\ef)\tfrac{d}{dt}\varphi\big|_{t=0},
\] 
for all $\varphi\in \Hom_c(V_F,\Lie(\dT)^{\Gamma_\ef})$ and $\omega\in\we$ 
(see also Proposition \ref{prop:ExplicitCocycle0}).
Both sides are complex vector spaces with their complex structures inherited by that of $\dT$, so the map thus defined is $\bbc$-linear.
In the $\bbz$-cases, given that we have with a 
semisimple $\bbc$-linear action of $V_F$ on a finite-dimensional vector space, $\uH^1(V_F,\Lie(\dT)^{W^1_\ef})$ is naturally homeomorphic to the $V_F$-coinvariants in $\Lie(\dT)^{W^1_\ef}$, hence to the space of invariants $\Lie(\dT)^{\Gamma_\ef}$.
\end{proof}

\subsection{\texorpdfstring{The identity component of $H^1_c(W_{E/F},\dT)$}{The identity component of H1c(W{E/F},T)}}
In this subsection, we apply our knowledge of the low-dimensional cohomology groups $\uH^r(W_\ef, A)$ with $A= X, A= \Lie(\dT)$ or $A= \dT$ to some natural constructions in homological algebra. In particular, we use the topological features of Moore's measurable cohomology to extract information about the structure and connected components of the resulting cohomology groups.

Since $X = \Hom(T_E,\bbg_m)$ is a free abelian group, applying $\Hom(\dX,-)$ to the injective resolution $0\to\bbz\to\bbc\to\bbc^*\to 0$ of the trivial module yields the short exact sequence
\begin{equation}\label{eq:FFSES}
0 \to X \to \Lie(\dT)\to\dT\to 0.
\end{equation}
Using Proposition \ref{prop:MooreSummary}(d),
we get the long exact sequence 
\begin{multline}\label{eq:LES}
    0\to X^\ga\to\Lie(\dT)^\ga\to\dT^\ga\to\\
    \to \uH^1(W_\ef,X)\to\uH^1(W_\ef,\Lie(\dT))\to\uH^1(W_\ef,\dT)\to \uH^2(W_\ef,X),
\end{multline}
where all the arrows are continuous. The first conclusion that we draw from this is the following.
\begin{prop}\label{prop:ConCompOp}
    If $F$ is a local or global field, then  $H_c^1(W_\ef,\dT)^\circ$ is an open subgroup of $H_c^1(W_\ef,\dT)$.
\end{prop}

\begin{proof}
    We use Propositions \ref{prop:MooreSummary}(c) and that $\uH^1(W_\ef, \dT)^\circ$ is the inverse image of the open set $\{0\}\subset\uH^2(W_\ef,X),$ from  Proposition \ref{prop:DiscreteIsDiscrete}.
\end{proof}

\begin{thm}\label{thm:Rcase}
Suppose that $V_F = \bbr_+$. Then $H_c^1(W_\ef,\dT)^\circ \approx \Lie(\dT^{\ga})$. In particular, $H_c^1(W_\ef,\dT)$ is an abelian Lie group.
\end{thm}

\begin{proof}
     Consider the short exact sequence $0\to W^1_\ef\to W_\ef \to V_F\to 0$, where the last arrow is the absolute value map, as discussed in Section \ref{sec:absolutevals}. From the inflation-restriction sequence we obtain the exact sequence
    \[
0\to \uH^1(V_F,\dT^{W^1_\ef})\to\uH^1(W_\ef,\dT)\to\uH^1(W^1_\ef,\dT),
    \]
    in which all the arrows are continuous.
    Since $V_F=\bbr_+$ acts trivially on $\dT$, using Proposition \ref{prop:MooreSummary}(c), we have $\bbc$-linear identifications \[\uH^1(V_F,\dT^{W^1_\ef})\approx \Hom_c(V_F,\dT^{\ga}) \approx \Lie(\dT^{\Gamma_\ef}).\] We now claim that $\uH^1(W^1_\ef,\dT)$ is discrete, as a topological space. Together with the continuity of $\uH^1(W_\ef,\dT)\to\uH^1(W^1_\ef,\dT)$ and Proposition \ref{prop:MooreSummary}(c), the claim implies that we have a continuous bijective homomorphism 
    \[
    \Lie(\dT)^\ga\to H_c^1(W_\ef,\dT)^\circ
    \]
    between locally compact Hausdorff groups (Theorem \ref{thm:Hausdorff}) with domain a finite-dimensional real vector space. Hence it is a homeomorphism, by the open mapping theorem for topological groups (see \cite[Theorem (5.29)]{HR}).
    
    As for the claim, since $W^1_\ef$ is compact, it follows from Proposition \ref{prop:MooreSummary}(g) that $\uH^p(W^1_\ef,\Lie(\dT))=0$ for $p=1,2$. From the long exact sequence of $W^1_\ef$-cohomology applied to 
    $0 \to X \to \Lie(\dT) \to \dT \to 0,$
    we thus obtain a continuous isomorphism of abelian groups
    \[
    \uH^1(W^1_\ef,\dT) \cong \uH^2(W^1_\ef,X).
    \] But from Proposition \ref{prop:MooreSummary}(f), it follows that $\uH^2(W^1_\ef,X)$ is countable and discrete as a topological space and, hence, so is $\uH^1(W^1_\ef,\dT)$, as claimed. 

    The last assertion follows from the above and Proposition \ref{prop:ConCompOp}.
    \end{proof}

We now turn our attention to the cases when $V_F\cong \bbz$. In these cases, the Weil group $W_\ef$ is not almost connected. We start by noting that the long exact sequence \eqref{eq:LES} yields the exact sequence
\begin{equation}\label{eq:FFSES2}
0\to A \to \uH^1(W_\ef,\Lie(\dT))\to\uH^1(W_\ef,\dT)^\circ\to 0,
\end{equation}
where $A\subseteq  \Lie(\dT)^\ga$ is the image of the continuous homomorphism $\uH^1(W_\ef, X)\to\uH^1(W_\ef,\Lie(\dT)) \approx\Lie(\dT)^\ga$. By Proposition \ref{prop:DiscreteIsDiscrete} we 
conclude that $A$ is countable, and since $A$ is also the kernel of the continuous map 
 $\Lie(\dT)^\ga\to\uH^1(W_\ef,\dT)^\circ$ 
we also see that $A$ is closed. A closed, countable subgroup of a finite-dimensional real or complex vector space is necessarily a discrete subgroup. 
Our next tasks will be to endow the space $\uH^1(W_\ef,\dT)^\circ$ with a canonical complex structure and to study the discrete group $A$. 
\begin{defn}\label{def:CanComplexStruc}
    For $F$ local or global,
    we endow the abelian Lie group $\uH^1(W_\ef,\dT)$ with a canonical structure of abelian complex Lie group via the sequence
    \begin{equation}\label{eq:CanCplxStructure}
    0 \to K \to \uH^1(W_\ef,\Lie(\dT)) \to \uH^1(W_\ef,\dT)^\circ \to 0,
    \end{equation}
    where $K$ is the kernel of the functorial homomorphism $\uH^1(W_\ef,\Lie(\dT)) \to \uH^1(W_\ef,\dT)$. We transport this complex structure to $H^1_c(W_\ef,\dT)$ via the isomorphism of Proposition \ref{prop:MooreSummary}(c).
\end{defn}
To justify Definition \ref{def:CanComplexStruc}, first of all, note that $\uH^1(W_\ef,\dT)$ is indeed an abelian Lie group, from Theorem \ref{thm:Rcase} and \eqref{eq:FFSES2}. Secondly, the middle term $\uH^1(W_\ef,\Lie(\dT))$ has a canonical structure of a complex vector space since $\Lie(\dT)$ is a complex vector space. Further, in the $\bbr$-case, the kernel $K=0$ by Proposition \ref{prop:H1middle} and Theorem \ref{thm:Rcase} and the fact that any continuous homomorphism between real vector spaces is necessarily  linear. In the $\bbz$-case, $K=A$, the discrete space appearing in \eqref{eq:FFSES2}. Hence, \eqref{eq:CanCplxStructure} endows $\uH^1(W_\ef,\dT)^\circ$,  and thus $\uH^1(W_\ef,\dT)$, with a canonical complex Lie group structure.

\begin{prop}
    The topological isomorphism $\Lie(\dT^{\ga})\approx H_c^1(W_\ef,\dT)^\circ$ of Theorem \ref{thm:Rcase} is a $\bbc$-linear isomorphism.
\end{prop}

\begin{proof}
Recall that $V_F\cong \bbr$. We have the following commutative diagram
    \begin{equation}\label{eq:RcaseDiagram}
\begin{tikzcd}
\Lie(\dT)^\ga \arrow[r,"\approx"]\arrow[d,"\approx"]& H^1_c(V_F, \Lie(\dT)^\ga) \arrow[r,"\approx","\textup{Inf}"'] \arrow[d,"\approx"] & H^1_c(W_{E/F}, \Lie(\dT))\arrow[d,twoheadrightarrow]  \\
\Lie(\dT^\ga) \arrow[r,"\approx"]&H^1_c(V_F, \dT^\ga) \arrow[r,twoheadrightarrow,,"\textup{Inf}"] & H^1_c(W_{E/F}, \dT)^\circ
\end{tikzcd}
\end{equation}
where the top row is discussed in (the proof of) Proposition \ref{prop:H1middle}, the bottom row is discussed in (the proof of) Theorem \ref{thm:Rcase} and the last two vertical arrows are the functorial homomorphisms associated to the exponential map. Since the bottom left arrow, the leftmost vertical arrow and the arrows 
in the top row are all $\bbc$-linear isomorphism, whereas the rightmost vertical arrow defines the complex structure on $H^1_c(W_{E/F}, \dT)^\circ$ (by Definition \ref{def:CanComplexStruc}), the result follows.
\end{proof}

\begin{prop}\label{prop:ALat}
If $V_F\cong \bbz$, let $A\subset\textup{Lie}(\widehat{T})^{\Gamma_{E/F}}$ be as in \eqref{eq:FFSES2}. Then,
we have a natural isomorphism $A\cong \Hom_c(\Hom_\ga(X,C_E),\bbz) = \Hom_c(T_{C_F},\bbz)$ and $A$ is a lattice in the real form $\fa_0^* = \bbr\otimes X^\ga$ of $\Lie(\dT)^\ga$.
\end{prop}

\begin{proof}
When we apply the functor $\Hom_c(T_{C_F},-)$ to the exponential sequence $0\to\bbz\to\bbc\to\bbc^*\to 0$ we obtain 
the exact sequence:
\begin{equation}\label{eq:homc1}
0\to \Hom_c(T_{C_F},\bbz)\to\Hom_c(T_{C_F},\bbc)\to \Hom_c(T_{C_F},\bbc^*).
\end{equation}
It was shown in \cite[Theorem 2.1]{Bi} (or see 
\cite[p. 245]{La} for the two special cases $D=\bbc$ and $D=\bbc^*$ which are relevant here) 
that for any divisible abelian topological group  $D$ with trivial $\Gamma_{E/F}$-action we have a natural isomorphism:  
\begin{equation}\label{eq:homc2}
H^1_c(\we,\Hom(\dX,D)) \xrightarrow{\cong}\Hom_c(\Hom_\ga(X,C_E),D).
\end{equation}
Recall from (\ref{eq:FFSES2}) that $A$ was defined as: 
\begin{equation}
A\cong \textup{Ker}(H^1_c(\we,\Hom(\dX,\bbc))\to H^1_c(\we,\Hom(\dX,\bbc^*))),
\end{equation}
when we use the isomorphism 
$H^1_c(\we,\Hom(\dX,\bbc))=H^1_c(\we,\textup{Lie}(\widehat{T}))\cong\textup{Lie}(\widehat{T})^{\Gamma_{E/F}}$ of Proposition \ref{prop:H1middle} to view $A$ as a lattice 
in the real form $\fa_0^*$ of $\Lie(\widehat{T})^{\Gamma_{E/F}}$. 
Combining this with (\ref{eq:homc1}) and (\ref{eq:homc2}) 
proves the desired result. 
\end{proof}
\begin{cor}\label{cor:Amax}
Suppose that $V_F\cong \bbz$. Then $H_c^1(\we,\dT)^\circ\approx H_c^1(\we,\Lie(\dT))/A$ is a complex algebraic torus, equipped with a canonical surjective homomorphism \[H_c^1(\we,\dT)^\circ\to X_T\] 
with finite kernel (which is an isomorphism in the local case). 
\end{cor}

\begin{proof}
    For both $F$ local non-Archimedean and a global function field, we can characterize $X_T$ via the exact sequence (see Proposition \ref{prop:tildeXT} and Corollary \ref{cor:GlobXT})
\begin{equation}\label{eq:XTSES}
0\to \widehat{L}_T \to \Lie(\dT)^\ga \to X_T \to 0,
\end{equation}
with $L_T$ the image of the respective $\log_T$ maps (see \eqref{eq:LogMap} and the text above Corollary \ref{cor:GlobXT}), and $\widehat{L}_T$ is the lattice $\Hom(L_T,\bbz)$.
Now, when $F$ is non-Archimedean local, we have $T_{C_F}=\Hom_\ga(X,C_E) = T(F)$. From continuity, it follows that
$\Hom_c(T(F),\bbz) = \Hom_c(T(F)/T(F)^1,\bbz)$ and hence $A\cong \widehat{L}_T$, using Proposition \ref{prop:ALat}. From  
(\ref{eq:FFSES2}) and (\ref{eq:XTSES}) we get $H_c^1(W_\ef,\dT)^\circ\approx X_T$. 

For $F$ a function field, we have 
$T(\bba_F)/T(F) \to T_{C_F}$, a closed inclusion of finite index (see \cite[p. 245]{La}) from which we obtain (by Proposition \ref{prop:ALat}) an inclusion 
$A \hookrightarrow \Hom_c(T(\bba_F)/T(F),\bbz) = \Hom(T(\bba_F)/T(\bba_F)^1,\bbz) \cong \widehat{L}_T$
with finite cokernel. 
By (\ref{eq:XTSES}) this implies the result. 
\end{proof}

\section{Continuity of Langlands's map}\label{sec:LangCont}

In this section, we shall finish the proof of our main result, Theorem \ref{thm:mainThm}. Recall the notation
$T_{C_F} := \Hom_{\Gamma_\ef}(X,C_E)$ of \eqref{eq:TCF}. From the exponential sequence $0\to\bbz\to\bbc\to\bbc^*\to 0$, 
on the one hand, applying the functor $\Hom_c(T_{C_F},-)$  we obtain the exact sequence
\begin{equation}\label{eq:TCFES}
 0 \to \Hom_c(T_{C_F},\bbz) \to \Hom_c(T_{C_F},\bbc) \to \Hom_c(T_{C_F},\bbc^*)   
\end{equation}
while on the other hand we get the sequences \eqref{eq:FFSES} and \eqref{eq:LES}. From \cite{La} and \cite[Proposition 2.15]{Bi}, there are natural maps between these sequences as follows: 
\begin{confidential}
    these sequences are connected at the level of continuous homomorphisms and continuous cocycles via the commuting diagram
\end{confidential}
\begin{equation}\label{eq:TCFcommdiagram}
    \begin{tikzcd}
        \Hom_c(T_{C_F},\bbc) \arrow[r]         & \Hom_c(T_{C_F},\bbc^*)\\
        H^1_c(W_\ef,\Lie(\dT)) \arrow[r]\arrow[u]& H^1_c(W_\ef,\dT)\arrow[u,"\Lambda"]
    \end{tikzcd}
\end{equation}
where the vertical arrows are isomorphisms (see \cite[Theorem 2.1]{Bi}) and $\Lambda$ is Langlands's canonical map. The commutativity of the diagram \eqref{eq:TCFcommdiagram}   can be checked by unravelling the identifications in Langlands's original results, as explained in \cite[Section 2]{Bi} (see, in particular, \cite[Proposition 2.4]{Bi} and \cite[Remark 2.5]{Bi} where it is evident that an arrow $D\to D'$ gives rise to a commutative diagram).

With the compact-open topologies, Moore's results (Proposition \ref{prop:MooreSummary}) imply that the horizontal arrows are continuous. As a consequence of Propositions \ref{prop:charTCFconcomp} and \ref{prop:ConCompOp}, which show that the identity components of the domain and codomain of $\Lambda$ are open, it suffices to study the continuity of $\Lambda$ when restricted to those components. Recall from Proposition \ref{prop:H1middle} that we have 
isomorphisms
\[
H^1_c(W_\ef,\Lie(\dT))\approx \Hom_c(V_F,\Lie(\dT)^{\Gamma_\ef}) \approx\Lie(\dT)^{\Gamma_\ef}.
\]
There is a similar characterization of $\Hom_c(T_{C_F},\bbc)$:
\begin{prop}\label{prop:Homcmiddle}
    For $F$ local or global,
    we have 
    $\bbc$-linear isomorphisms
    \[
    \Hom_c(T_{C_F},\bbc) \approx 
    \Lie(\dT)^{\Gamma_\ef}.
    \]
\end{prop}

\begin{proof}
    First,  note that as $T_{C_F}^1=\Hom_{\Gamma_\ef}(X,C_E^1)$ is the maximal compact subgroup of $T_{C_F}$ (see Proposition \ref{prop:MaxCompact}), the continuity implies a natural identification $\Hom_c(T_{C_F},\bbc)\approx \Hom_c(T_{C_F}/T_{C_F}^1,\bbc)$. The result follows from the characterization of the quotient $T_{C_F}/T_{C_F}^1$ given in Proposition \ref{prop:tildeXT}. 
\end{proof}
    
\begin{thm}\label{thm:ApexThm}
    For $F$ local or global, the natural homomorphism 
    \[
    \Lambda:H^1_c(W_\ef,\dT) \to \Hom_c(T_{C_F},\bbc^*)
    \]
    described by Langlands is an isomorphism of abelian complex Lie groups.
\end{thm}

\begin{proof}
    As remarked before, to discuss the continuity of the Langlands map, it suffices to restrict to the identity component. 
    
    Recall from Proposition \ref{prop:charTCFconcomp}, that we have an identification $\Hom_c(T_{C_F},\bbc^*)^\circ\approx\Hom_c(T_{C_F}/T_{C_F}^1,\bbc^*)$. 
    We also have $\Hom_c(T_{C_F},A)\approx \Hom_c(T_{C_F}/T^1_{C_F},A)$ when $A=\bbz$ or $A=\bbc$. We claim that \eqref{eq:TCFES} becomes a short exact sequence upon replacing $\Hom_c(T_{C_F},\bbc^*)$ by $\Hom_c(T_{C_F},\bbc^*)^\circ$. Indeed, $T_{C_F}/T^1_{C_F}$ was shown to be either a lattice or a finite-dimensional real vector space in Proposition \ref{prop:tildeXT}. 
    Thus, the claim reduces to the cases $T_{C_F}/T^1_{C_F}\approx\bbr$ and $T_{C_F}/T^1_{C_F}\approx\bbz$, 
    both of which are trivial. We thus obtain a commutative diagram, below, where the rightmost column is the sequence just described:
    \begin{equation}\label{eq:ApexDiagram}
\begin{tikzcd}
0 \arrow[d]                                                    &0\arrow[d]\\
A_F \arrow[d]\arrow[r]                                         &\Hom_c(T_{C_F},\bbz) \arrow[d]\\
H_c^1(W_\ef,\Lie(\dT)) \arrow[d]\arrow[r, "\textup{d}\Lambda"] &\Hom_c(T_{C_F},\bbc) \arrow[d]\\
H_c^1(W_\ef,\dT)^\circ \arrow[d]\arrow[r, "\Lambda"]           &\Hom_c(T_{C_F},\bbc^*)^\circ\arrow[d]\\
0                                                              &0
\end{tikzcd}
.
\end{equation}
The commutativity of \eqref{eq:ApexDiagram} follows from the commutativity of \eqref{eq:TCFcommdiagram} and the column in the left is the long exact sequence in cohomology \eqref{eq:LES} obtained from the short exact sequence $0\to X \to \Lie(\dT)\to\dT\to 0$, when truncated with respect to the identity component of $H_c^1(W_\ef,\dT)$. Here, the module $A_F$ is obtained by modding-out the torsion submodule of $H_c^1(W_\ef,X)$. When $V_F\cong \bbr$ we have $A_F=0$, while when $V_F\cong \bbz$, $A_F$ is the lattice described in Proposition \ref{prop:ALat}. 

By Propositions \ref{prop:H1middle} and \ref{prop:Homcmiddle}, both sides of the middle row are  finite-dimensional complex vector spaces, while  $\Lambda$ is a homomorphism between (connected) abelian Lie groups. Moreover, each column is the universal cover of its respective group. 
To prove our result,
it is enough to show that the map $\textup{d}\Lambda$ is $\bbc$-linear. To do so, the key step is the $\bbz$-bilinear pairing 
of \cite[Remark 2.5]{Bi}, which for \( D = \bbc \), yields the $\bbc$-bilinear pairing 
\(
H^1(W_{\ef},\Lie(\dT)) \times H_1(W_{\ef},\dX) \to \bbc
\) 
and hence the $\bbc$-linear homomorphism
\begin{equation}\label{eq:TheLinearMap}
H^1(W_{\ef},\Lie(\dT)) \to \Hom(H_1(W_{\ef},\dX),\bbc).
\end{equation}  
Through the chain of identifications   
\[
H_1(W_{\ef},\dX) \cong H_1(C_E,\dX)^{\Gamma_{\ef}} \cong \Hom_{\Gamma_{\ef}}(X,C_E) = T_{C_F}
\]  
(see \cite[p. 241]{La} and \cite[Propositions 2.3 and 2.8]{Bi}), $\textup{d}\Lambda$ is the restriction of \eqref{eq:TheLinearMap} to continuous classes (see \cite[Proposition 2.15]{Bi}), hence $\bbc$-linear, as required. Since the identity components on both sides of $\Lambda$ are open, we are done.
\end{proof}

Combined with the fact that $\Hom_c(T_{C_F},\bbc^*)\to \Hom_c(T(\bba_F)/T(F),\bbc^*)$ is open (Proposition \ref{prop:OpClosed}), Theorem \ref{thm:mainThm} is proved.

\begin{prop}\label{prop:H2X}
    For $F$ local or global, we have $\uH^2(W_\ef,X)\approx \Hom_c(T_{C_F}^1,\bbc^*)$.
\end{prop}

\begin{proof}
    It is known that $\uH^2(W_F,\dT)=0$ (see \cite[Theorem 2]{Ra}), where $W_F$ is the absolute Weil group of $F$. From the extension $0\to W_E^c\to W_F\to W_\ef\to 0$, the inflation-restriction sequence implies that
    \[
    \uH^1(W_E^c,\dT)^{W_\ef}\to \uH^2(W_\ef,\dT)\to \uH^2(W_F,\dT)
    \]
    is exact, and since $\uH^2(W_F,\dT)=0$, the first map is surjective. 
    Recall that $W_E^c$ is compact, and that 
    the action of $W_E$ on $\dT$ is trivial. In particular we  have 
    $\uH^1(W_E^c,\dT) \approx \Hom_c(W_E^c,\dT)$. Taking the long exact sequence in Moore's measurable cohomology theory for the compact group $W_E^c$ associated to the short exact sequence (\ref{eq:FFSES}), and using Proposition 
    \ref{prop:MooreSummary}(f) and (g), we obtain that $\uH^1(W_E^c,\dT)\approx \uH^2(W_E^c,X)$ is countable and discrete. It follows that $\uH^2(W_\ef,\dT)$ is countable.
    With this as input, consider the long exact sequence in  Moore's measurable cohomology for $W_{E/F}$ for the short exact sequence (\ref{eq:FFSES}). 
    We see that $\uH^2(W_\ef,\Lie(\dT))$, which is a priori a 
    real vector space, is squeezed between $\uH^2(W_\ef,X)$ (a countable discrete abelian group by Proposition \ref{prop:DiscreteIsDiscrete}), 
    and $\uH^2(W_\ef,\dT)$ (a countable abelian group by the above). Hence, $\uH^2(W_\ef,\Lie(\dT))$ is $0$.
    From the isomorphism $\Lambda$, Proposition \ref{prop:H1middle} and Proposition \ref{prop:charTCFconcomp} 
    it now follows that:
    \[\uH^2(W_\ef,X)\approx\frac{\uH^1(W_\ef,\dT)}{\uH^1(W_\ef,\dT)^\circ} \approx
    \frac{\Hom_c(T_{C_F},\bbc^*)}{\Hom_c(T_{C_F},\bbc^*)^\circ}\approx\frac{\Hom_c(T_{C_F},\bbc^*)}{\Hom_c(T_{C_F}/T_{C_F}^1,\bbc^*)},  \]
    from which we conclude
    $\uH^2(W_\ef,X)\approx \Hom_c(T_{C_F}^1,\bbc^*).$
\end{proof}

\begin{rem}\label{rem:Schwein}
    Using the topological isomorphism of Proposition \ref{prop:LocTriv} and considering the long exact sequence \eqref{eq:FFSES} with $W_{\ef}$ replaced by $W_F$ (where we also have $\uH^2(W_F,\Lie(\dT)) = 0$), we obtain an isomorphism $\uH^2(W_\ef,X) \approx \uH^2(W_F,X)$. The latter was the formulation used by Schwein in \cite[Section 5.2]{Sc} (in the local non-Archimedean case).
\end{rem}

\begin{thm}\label{thm:Zcase}
Suppose that $V_F\cong \bbz$. If $F$ is local non-Archimedean, then,  $H_c^1(\we,\dT)^\circ$ is isomorphic to $X_T$. When $F$ is a global function field, then the group $H_c^1(\we,\dT)^\circ/\sim_{l.e.}$
is isomorphic to $X_T$. 
\end{thm}

\begin{proof}
See Corollary \ref{cor:Amax} for the proof in the non-Archimedean local cases. For global function fields, the proof of Corollary \ref{cor:Amax} yields a commutative diagram (in which the rightmost vertical arrow is a surjection with finite kernel):
\begin{equation}\label{eq:XTdiagrams}
\begin{tikzcd}
0 \arrow[r] & A \arrow[r] \arrow[d,hookrightarrow ]   & H^1_c(W_\ef,\Lie(\dT)) \arrow[r]\arrow[d,"\approx"] &H_c^1(W_\ef,\dT)^\circ \arrow[r]\arrow[d,twoheadrightarrow]& 0\\
0 \arrow[r] & \widehat{L}_T \arrow[r]                                   & \Lie(\dT)^\ga \arrow[r]                           &X_T \arrow[r]                                                                           & 0
\end{tikzcd}
.
\end{equation}
We claim that the kernel of $H_c^1(W_\ef,\dT)^\circ\to X_T$ consists of locally trivial classes. For each local place $v\in\Pl(F)$, let us denote by $E_{v'} = EF_v$ the induced completion of $E$ and denote by $\Gamma_{v'|v}$ and $W_{v'|v}$ the Galois group and the relative Weil group of the extension $E_{v'}/F_v$. Let also $L_{T,v}\approx T(F_v)/T(F_v)^1$ denote the image of the local $\log_T$ map, and let $X_{T,v}$ denote the space of local unramified characters.
Using the isomorphism between $X_{T,v}$ and $H_c^1(W_{v'|v},\dT)^\circ$ of Corollary \ref{cor:Amax}, we obtain a commutative diagram 
\[
\begin{tikzcd}
0 \arrow[r] & \widehat{L}_{T,v} \arrow[r]                                   & \Lie(\dT)^{\Gamma_{v'|v}} \arrow[r]           &X_{T,v} \arrow[r]                                              & 0\\
0 \arrow[r] & \Hom_c(T(F_v),\bbz) \arrow[r] \arrow[u, "\approx"]   & H_c^1(W_{v'|v},\Lie(\dT)) \arrow[r]\arrow[u,"\approx"]  &H_c^1(W_{v'|v},\dT)^\circ \arrow[r]\arrow[u,"\approx"]& 0\\
0 \arrow[r] &  A\arrow[r] \arrow[d,hookrightarrow ]\arrow[u]   & H_c^1(W_\ef,\Lie(\dT)) \arrow[r]\arrow[d,"\approx"]\arrow[u]  &H_c^1(W_\ef,\dT)^\circ \arrow[r]\arrow[d,twoheadrightarrow]\arrow[u]& 0\\
0 \arrow[r] & \widehat{L}_T \arrow[r]                                   & \Lie(\dT)^\ga \arrow[r]                          &X_T \arrow[r]                                                                           & 0
\end{tikzcd}
.
\]
The vertical arrows on the top and bottom right-hand side are the restriction to the identity components of the local and global Langlands maps respectively, and the map in the middle is the natural assignment from a global to local Langlands parameters. We can further visualize the rightmost column via the commutative diagram
\[
\begin{tikzcd}
H^1_c(W_{E/F}, \dT)^\circ \arrow[r] \arrow[d,twoheadrightarrow] & H^1_c(W_{v'|v}, \dT)^\circ \arrow[d] \\
X_T \arrow[r] & X_{T,v}
\end{tikzcd}
\]
defined for each place $v$ of $F$. Making use of the canonical map 
\[\prod_v\Res_v:H^1_c(W_{F}, \dT) \to \prod_v H^1_c(W_{F_v}, \dT)\] 
of Proposition \ref{prop:LocTriv}, after identifying the cohomology groups of the absolute and relative Weil groups using the inflation maps and restricting to the connected components, we obtain a well-defined map
\[
X_T \to  \prod_v X_{T,v}.
\]
Now, given a global parameter $\phi\in H_c^1(W_\ef,\dT)^\circ$, let $\phi_{v}$ denote the corresponding local parameters of $\phi$, and let  
$\chi_\phi\in X_T$ denote 
the corresponding automorphic character of $T$. Since automorphic characters are completely determined by their local components, since Langlands's maps are compatible with localization to a local place, and since the local Langlands correspondences are bijective, 
it follows that $\chi_\phi=1$ if and only if all local 
parameters $\phi_v$ are trivial. This finishes the proof. 
\end{proof}

With Theorem \ref{thm:Zcase} at hand, we can recover the explicit description of $X_T$ from \cite[Section 3.3.1]{Ha}, when $F$ is local non-Archimedean. To see that, recall that in this case, we have the short exact sequence $0\to I_F\to W_F\to \lpi\textup{Fr}\rpi \to 0$, where $W_F$ is the absolute Weil group of $F$, $I_F$ is the inertia subgroup, $\textup{Fr}$ is a Frobenius automorphism and $\lpi \textup{Fr} \rpi\approx\bbz$. Since the arrow $W_F\to \lpi \textup{Fr} \rpi$ factors through $W_\ef \to \lpi \textup{Fr} \rpi$, as the continuous image of the closed commutator group $W_E^c$ has a trivial image, we obtain the short exact sequence
\begin{equation}\label{eq:InertiaSeq}
    0 \to N_F \to W_\ef \to \lpi \textup{Fr} \rpi \to 0,
\end{equation}
where $N_F$ is the image of $I_F$ in the quotient $W_F/W_E^c\cong W_\ef$. It is well known that $N_F$ is compact, as $I_F$ is profinite \cite[Section (1.4.1)]{Ta1}.  

\begin{prop}[\cite{Ha}]\label{prop:LnAHaines}
    When $F$ is local non-Archimedean, the space of unramified characters $X_T$ is isomorphic to $((\dT^{N_F})_{\textup{Fr}})^\circ$, 
    as complex algebraic tori.
\end{prop}

\begin{proof}
    From the short exact sequence $0\to X \to \Lie(\dT) \to \dT \to 0$, we obtain, from the long exact sequence in $N_F$-cohomology, the exact sequence
    \[
   \uH^1(N_F,\Lie(\dT))\to\uH^1(N_F,\dT)\to \uH^2(N_F,X)\to \uH^2(N_F,\Lie(\dT)).
    \]
Since $N_F$ is compact, from Proposition \ref{prop:MooreSummary}, items (f) and (g), we obtain that $\uH^1(N_F,\dT)\approx \uH^2(N_F,X)$ is countable and discrete, as a topological space. Now, from the short exact sequence \eqref{eq:InertiaSeq} and using the first few terms of the inflation-restriction sequence applied to the $W_\ef$-module $\dT$, we obtain the exact sequence
 \begin{equation}\label{eq:localXTiso1}
    0\to \uH^1(\lpi \textup{Fr} \rpi,\dT^{N_F}) \to \uH^1(W_\ef,\dT) \to
    \uH^1(N_F,\dT)^{\langle \textup{Fr}\rangle}
\end{equation}
where all the arrows are continuous.  As $ \lpi \textup{Fr} \rpi\approx \bbz$ 
is discrete, we know that the measurable, the continuous and the abstract cohomology theories coincide, from which we obtain that \[\uH^1(\lpi \textup{Fr} \rpi,\dT^{N_F}) = H_c^1(\lpi \textup{Fr} \rpi,\dT^{N_F})= H^1(\lpi \textup{Fr} \rpi,\dT^{N_F}) = (\dT^{N_F})_{\textup{Fr}}\] is a complex diagonalizable group (recall that the action of $W_\ef$ on $\dT$ is via a finite group of automorphisms of complex algebraic tori). 
Since $\uH^1(W_\ef,\dT)$ is a complex Lie  group (by Theorem \ref{thm:Rcase}, \eqref{eq:FFSES2} and Definition \ref{def:CanComplexStruc}) and $\uH^1(N_F,\dT)^{\langle \textup{Fr}\rangle}$ is discrete, 
the exactness of \eqref{eq:localXTiso1}
implies that we have a continuous  isomorphism of Lie groups between
$\uH^1(\lpi \textup{Fr} \rpi,\dT^{N_F})$ and $K\subset \uH^1(W_\ef,\dT)$, the kernel of $\uH^1(W_\ef,\dT) \to \uH^1(N_F,\dT)^{\langle \textup{Fr} \rangle}$, which a closed and open subgroup of $\uH^1(W_\ef,\dT)$. Hence, $\uH^1(\lpi \textup{Fr} \rpi,\dT^{N_F}) \to \uH^1(W_\ef,\dT)$ is an injective open homomorphism of Lie groups, from which we conclude that 
\[
((\dT^{N_F})_{\textup{Fr}})^\circ=\uH^1(\lpi \textup{Fr} \rpi,\dT^{N_F})^\circ \approx \uH^1(W_\ef,\dT)^\circ\approx X_T
\]
are isomorphisms of Lie groups. From the commutative diagram
\[
\begin{tikzcd}
    \uH^1(\lpi \textup{Fr} \rpi, \Lie(\dT)^{N_F})\arrow[r,equal]\arrow[d,]&\Lie(\dT)^{\ga} \arrow[r,"\approx"]\arrow[d]& \uH^1(\we,\Lie(\dT))\arrow[d,]\\
        \uH^1(\lpi \textup{Fr} \rpi, \dT^{N_F})^\circ\arrow[r, equal]&((\dT^{N_F})_\textup{Fr})^\circ \arrow[r,]& \uH^1(\we,\dT)^\circ
    \end{tikzcd},
\]
since the top-right arrow is the $\bbc$-linear isomorphism of Proposition \ref{prop:H1middle} and the leftmost vertical  arrow is the one defining the complex structure on $\uH^1(\we,\dT)$ (see Definition \ref{def:CanComplexStruc}),
the isomorphism $\uH^1(\lpi \textup{Fr} \rpi,\dT^{N_F})^\circ \approx \uH^1(W_\ef,\dT)^\circ$ preserves the complex structures, and we are done.
\end{proof}

\section{Explicit Cocycles; Proof of Main Theorem \ref{thm:mainThm2}}\label{sec:ExplicitCocycle}

In this last section, we exhibit an explicit realization of certain cocycles in $Z^1_c(W_\ef,A)$, and thus cohomology classes in $H_c^1(W_\ef,A)$, with $A$ being either $\Lie(\dT)$ or $\dT$.  The information provided here will be used to finalize the proof of Theorem \ref{thm:mainThm2}.

Let us retain the convention of Section \ref{s:unrchar} and write $\log_q$ where $q=e$ if $V_F= \bbr_+$ and $q=q_F$ if $V_F =q_F^{\bbz}$. Given any element $\nu\in \Lie(\dT)^{\Gamma_\ef}$, note that it determines a continuous $1$-cochain $\zeta_\nu:W_\ef\to \Lie(\dT)$  given by
\begin{equation}\label{eq:ExplicitCocycle}
\zeta_\nu(\omega)= (\log_{q} |\omega |_\ef)\nu,
\end{equation}
where $\omega\in W_\ef$. In other words, the assignment $\nu\mapsto \zeta_\nu$ determines a function
\[
\zeta:\Lie(\dT)^{\Gamma_\ef}\to C^1_c(W_\ef,\Lie(\dT))
\]
(where we recall that $C^1_c(W_\ef, \Lie(\dT))$ denotes the space of continuous $1$-cochains with values in $\Lie(\dT)$), which is, in fact, a $\bbc$-linear homomorphism since for any $\omega\in W_\ef$ we have
\[
    \zeta_{\nu+\nu'}(\omega) = (\log_{q} |\omega |_\ef)(\nu+\nu')=\zeta_\nu(\omega)+\zeta_{\nu'}(\omega)
    \]
for all $\nu,\nu'\in \Lie(\dT)^\ga$, and for each $\lambda \in \bbc$, 
    \[
    \zeta_{(\lambda\nu)}(\omega) = (\log_{q} |\omega |_\ef)(\lambda\nu) =  \lambda (\zeta_{\nu}(\omega)) = (\lambda \zeta_{\nu})(\omega).
    \]

\begin{prop}\label{prop:ExplicitCocycle0} 
The image of  $\zeta$ is inside $Z^1_c(W_\ef,\Lie(\dT))$ and it induces a $\bbc$-linear isomorphism 
\(
\Lie(\dT)^{\Gamma_\ef}\to H_c^1(W_\ef,\Lie(\dT)). 
\) 
\end{prop}

\begin{proof}
    If $\nu\in \Lie(\dT)^{\ga}$, then 
    \(\zeta_\nu(\omega_1\omega_2) = \zeta_\nu(\omega_1) + \zeta_\nu(\omega_2) = \zeta_\nu(\omega_1)+{}^{\overline{\omega_1}}\zeta_\nu(\omega_2),\)
    so that $\zeta_\nu\in Z^1_c(\we,\Lie(\dT))$. 
    Hence, it defines a homomorphism \[[\zeta]:\Lie(\dT)^\ga\to H^1_c(\we,\Lie(\dT))\] which is injective  since if $\zeta_\nu$ is a coboundary, that is, there exists $\mu\in\Lie(\dT)$ such that
    \[
    \zeta_\nu(\omega) = \log_q(|\omega|)\nu = {}^{\bar\omega}\mu - \mu
    \]
    for all $\omega\in W_\ef$, then by applying the norm $N_\Gamma = \sum_{\gamma}\gamma$ of the finite Galois group, we obtain
    \(
    N_\Gamma(\zeta_\nu(\omega)) = |\ga|\log_q(|\omega|)\nu = 0
    \)
    implying $\nu = 0$. By Proposition \ref{prop:H1middle} it follows that
    \([\zeta]:\Lie(\dT)^{\Gamma_\ef}\to H^1_c(W_\ef,\Lie(\dT))\)
    is a $\bbc$-linear isomorphism.
\end{proof}

In the $\bbz$-cases, we have an anologue to  the explicit linear cocycles \eqref{eq:ExplicitCocycle} but now with values in $\dT$. Let $z:\dT^\ga\to C_c^1(\we,\dT)$ be the homomorphism defined by 
\begin{equation}
    s\mapsto z_s(\omega) = s^{\log_q(|\omega|_\ef)}
\end{equation}
for all $s\in \dT^\ga$ and $\omega\in \we$. As in (the proof of) Proposition \ref{prop:ExplicitCocycle0}, one checks that $z_s\in Z_c^1(\we,\dT)$ for all $s\in T^\ga$ and hence $z$ induces a continuous homomorphism of Lie groups $[z]:\dT^\ga\to H^1(\we,\dT)$.

\begin{prop}\label{prop:ExplicitCocycle} 
If $V_F\cong \bbz$, the homomorphism $[z]:\dT^\ga\to H^1(\we,\dT)$ is a homomorphism of complex Lie groups and we have continuous surjections 
\[(\dT^{\Gamma_\ef})^\circ\stackrel{[z]}{\longrightarrow} H^1_c(W_\ef,\dT)^\circ\longrightarrow X_T\longrightarrow \dT_{\Gamma_\ef}.
\]
\end{prop}

\begin{proof}
In the $\bbz$-cases, analogously to \eqref{eq:RcaseDiagram}, we have a commutative digram
\begin{equation}\label{eq:ZcaseDiagram}
    \begin{tikzcd}
    \Lie(\dT)^\ga\arrow[r,"\approx"]\arrow[d,"\exp"]&H^1_c(V_F, \Lie(\dT)^\ga) \arrow[r,"{[\zeta]}"]\arrow[d,"\exp_*"]& H^1_c(\we,\Lie(\dT))\arrow[d,"\exp_*"]\\
        \dT^\ga\arrow[r,"\approx"]&H^1_c(V_F, \dT^\ga) \arrow[r,"{[z]}"]& H^1_c(W_\ef,\dT)
    \end{tikzcd}
\end{equation}    
which implies that $[z]$ preserves the canonical complex structures. By restricing to connected components in the bottom row, this justifies the first arrow. The surjectivity of the arrow $H^1_c(W_\ef,\dT)^\circ\to X_T$ was discussed in (\ref{eq:XTdiagrams}) and Corollary \ref{cor:Amax}. On the other hand, let $S\subset T$ be the largest split subtorus of $T$. Write $X(S) = \Hom(S,\bbg_m)$ for the character lattice of $S$ and $\dX(S)=\Hom(\bbg_m,S)$ for its cocharacter lattice. As $X^\ga\subseteq X(S)$ is an inclusion of finite index, it induces an inclusion of the respective images of the $\log$-maps (see text above Corollary \ref{cor:GlobXT})
    \(
    L_S \hookrightarrow L_T,
    \)
    which is also of finite index. Furthermore, as $S$ is split we have a canonical isomorphism $\dX(S)=\dX^\ga\cong L_S$, from which we obtain an epimorphism
    \[
    \Hom(L_T,\bbc^*)=X_T\to X_S = \Hom(L_S,\bbc^*) \cong \dT_\ga,
    \]
    finishing the proof.
\end{proof}
Together with Corollary \ref{cor:GlobXT} and Proposition \ref{prop:locT1}, this concludes the proof of Theorem \ref{thm:mainThm2}.


\begin{thebibliography}{9999}

\bibitem{AM} Austin, T., Moore, C.C.,
{\it Continuity properties of measurable group cohomology,} Mathematische Annalen, {\bf 356}(3) (2013), pp. 885--937.

\bibitem{Ar} Arens, R.F., 
{\it A topology for spaces of transformations,} Annals of Mathematics, {\bf 47}(3) (1946), pp. 480--495.

\bibitem{Bi} Birkbeck, C. D., 
{\it On the $ p $-adic Langlands correspondence for algebraic tori,} Journal de Th\'eorie des Nombres de Bordeaux {\bf 32}(1) (2020), pp. 133--158.

\bibitem{Bo} Borel, A.,
{\it Linear algebraic groups,}
Vol. 126. Springer Science \& Business Media, 2012.

\bibitem{Bo2} Borel, A., 
{\it Automorphic $L$-functions,} 
In: Automorphic forms, representations and $L$-functions (Proc. Sympos. Pure Math., Oregon State Univ., Corvallis, Ore., 1977), Part 2, AMS (1979), pp. 27--61.

\bibitem{Bou} Bourbaki, N.,
{\it Topologie G\'en\'erale, Chapitres 5 \`a 10,}
Hermann, Paris, 1974. 

\bibitem{Con} Conrad, B.,
{\it Finiteness of class numbers for algebraic groups,}\\
\url{https://virtualmath1.stanford.edu/~conrad/vigregroup/vigre02/cosetfinite.pdf}

\bibitem{Ha} Haines, T. J.,
{\it The stable Bernstein center and test functions for Shimura varieties},
Automorphic forms and Galois representations {\bf 2} (2014), pp. 118--186.

\bibitem{HR} Hewitt, E., K. A. Ross,
{\it Abstract Harmonic Analysis: Volume I Structure of Topological Groups Integration Theory Group Representations,}
A Series of Comprehensive Studies in Mathematics \textbf{115}. Springer Science \& Business Media, 2012.

\bibitem{Hu} Humphreys, J. E.,
{\it Linear algebraic groups,}
GTM \textbf{21}. Springer Science \& Business Media, 2012.

\bibitem{Lab} Labesse, J.-P.,
{\it Cohomologie, L-groupes et fonctorialit\'e,}
Compositio Math. {\bf 55}(2) (1985), pp. 163--184.

\bibitem{La} Langlands, R.P.,
{\it Representations of abelian algebraic groups,}
Pac J Math, \textbf{181}(3) (1997), pp. 231--250.

\bibitem{Ma} Mackey, G.,
{\it The Laplace Transform For Locally Compact Abelian Groups,}
Proceedings of the National Academy of Sciences \textbf{34}(4) (1948), pp. 156--162.

\bibitem{MW} Moeglin, C., Waldspurger, J.-L.,
{\it Spectral decomposition and Eisenstein series,}
Cambridge tracts in Mathematics \textbf{113}, Cambridge University Press, 1995.

\bibitem{Mo1} Moore, C.,
{\it Extensions and low dimensional cohomology theory of locally compact groups. I.}
Trans Amer Math Soc, \textbf{113}(1) (1964), pp. 40--63.

\bibitem{Mo2} Moore, C.,
{\it Extensions and low dimensional cohomology theory of locally compact groups. II.}
Trans Amer Math Soc, \textbf{113}(1) (1964), pp. 64--86.

\bibitem{Mo3} Moore, C.,
{\it Extensions and low dimensional cohomology theory of locally compact groups. III.}
Trans Amer Math Soc, \textbf{221}(1) (1976), pp. 1--33.

\bibitem{Mo4} Moore, C.,
{\it Extensions and low dimensional cohomology theory of locally compact groups. IV.}
Trans Amer Math Soc, \textbf{221}(1) (1976), pp. 35--58.

\bibitem{NSW} Neukirch, J., Schmidt, A. and Wingberg, K.,
{\it Cohomology of number fields,}
A Series of Comprehensive Studies in Mathematics \textbf{323}. Springer Science \& Business Media, 2013.

\bibitem{Pr} Prasad, G.,
{\it Elementary proof of a theorem of Bruhat-Tits-Rousseau and of a theorem of Tits.}
Bulletin de la Soci\'et\'e math\'ematique de France {\bf 110} (1982), pp. 197--202.

\bibitem{Ra} Rajan, C. S.,
{\it On the vanishing of the measurable Schur cohomology groups of Weil groups.}
Compositio Mathematica {\bf 140}(1) (2004), pp. 84--98.

\bibitem{RV} Ramakrishnan, D., Valenza, R.J.,
{\it Fourier analysis on number fields,}
Graduate Texts in Mathematics \textbf{186}. Springer Science \& Business Media, 1998.

\bibitem{Sc} Schwein, D.,
{\it Orthogonal root numbers of tempered parameters.}
Mathematische Annalen, {\bf 386}(3) (2023), pp. 2283--2319.

\bibitem{Ta1} Tate, J., 
{\it Fourier analysis in number fields and Hecke's zeta-functions.} 
In: Algebraic number theory, eds.: Cassels, J.W.S., Fr\"ohlich, A., Academic Press (1967), pp. 305--347.

\bibitem{Ta2} Tate, J., 
{\it Number theoretic background.} 
In: Automorphic forms, representations and $L$-functions (Proc. Sympos. Pure Math., Oregon State Univ., Corvallis, Ore., 1977), Part 2, AMS (1979), pp. 3--26.

\bibitem{Wa} Warner, F.W.,
{\it Foundations of differentiable manifolds and Lie groups,}
Vol. 94. Springer Science \& Business Media, 1983.

\bibitem{Wb} Weibel, C.A.,
{\it An introduction to homological algebra,}
No. 38. Cambridge university press, 1994.

\bibitem{We} Weil, A.,
{\it Basic number theory,}
Springer Science \& Business Media, 1974.

\bibitem{Yu} Yu, J.-K.,
{\it On the local Langlands correspondence for tori.}
In: Ottawa Lectures on Admissible Representations of Reductive \( p \)-adic Groups, eds.: Cunningham, C., Nevins, M., Fields Institute Monographs, AMS (2009), pp. 177--183.

\end{thebibliography}
\end{document}